\documentclass[a4paper,11pt,oneside  ]{article}	

\makeindex
\usepackage{hyperref}
\usepackage{graphicx}
\usepackage{enumerate}
\hypersetup{colorlinks, linkcolor=blue, urlcolor=red, filecolor=green, citecolor=blue}
\usepackage{amsmath}
\usepackage{amsthm}
\usepackage{amssymb}
\usepackage{authblk}
\usepackage{enumerate}
\usepackage{fullpage}
\usepackage{esint}

\theoremstyle{plain}
\newtheorem{theorem}{Theorem}[section]
\newtheorem{corollary}[theorem]{Corollary}
\newtheorem{lemma}[theorem]{Lemma}
\newtheorem{proposition}[theorem]{Proposition}

\theoremstyle{definition}
\newtheorem{definition}[theorem]{Definition}

\newtheorem{assumption}[theorem]{Assumption}

\theoremstyle{remark}
\newtheorem{remark}{Remark}

\newcommand{\N}{\mathbb{N}}
\newcommand{\R}{\mathbb{R}}

\newcommand{\Z}{\mathbb{Z}}

\newcommand{\Co}{{\operatorname{Co}}}

\newcommand\e{\varepsilon}

\newcommand\divv{\operatorname{div}}

\def\Xint#1{\mathchoice
   {\XXint\displaystyle\textstyle{#1}}%
   {\XXint\textstyle\scriptstyle{#1}}%
   {\XXint\scriptstyle\scriptscriptstyle{#1}}%
   {\XXint\scriptscriptstyle\scriptscriptstyle{#1}}%
   \!\int}
\def\XXint#1#2#3{{\setbox0=\hbox{$#1{#2#3}{\int}$}
     \vcenter{\hbox{$#2#3$}}\kern-.5\wd0}}

\def\fint{\Xint-}
\newcommand{\ol}{\overline}

\newcommand\dist{\operatorname{dist}}
\newcommand\sym{\operatorname{sym}}

\newcommand{\super}[1]{^{(#1)}}

\newcommand{\SO}[1]{\operatorname{SO}(#1)}
\newcommand\id{\operatorname{id}}
\newcommand\Id{\operatorname{Id}}

\newcommand{\wto}{\rightharpoonup}

\newcommand\ho{{\operatorname{hom}}}
\newcommand\per{{\operatorname{per}}}
\newcommand\loc{{\operatorname{loc}}}

\newcommand{\step}[1]{\medskip\noindent\textbf{Step #1. }}
\newcommand{\substep}[1]{\medskip\noindent\textit{Substep #1. }}

 \title{Quantitative homogenization in nonlinear elasticity for small loads}
 \author[1]{Stefan Neukamm\thanks{stefan.neukamm@tu-dresden.de}}
\author[1]{Mathias Sch\"affner\thanks{mathias.schaeffner@tu-dresden.de}}
\affil[1]{Department of Mathematics, Technische Universit\"at Dresden}

\begin{document}

\maketitle

\begin{abstract}
We study quantitative periodic homogenization of integral functionals in the context of non-linear elasticity. Under suitable assumptions on the energy densities (in particular frame indifference; minimality, non-degeneracy and smoothness at the identity; $p\geq d$-growth from below; and regularity of the microstructure), we show that in a neighborhood of the set of rotations, the multi-cell homogenization formula of non-convex homogenization reduces to a single-cell formula. The latter can be expressed with help of correctors. We prove that the homogenized integrand admits a quadratic Taylor expansion in an open neighborhood of the rotations -- a result that can be interpreted as the fact that homogenization and linearization commute close to the rotations. Moreover, for small applied loads, we provide an estimate on the homogenization error in terms of a quantitative two-scale expansion.
\medskip

\end{abstract}

\tableofcontents

\section{Introduction}

\subsection{Informal summary of results}

In this contribution we study quantitative homogenization of non-linearly elastic, periodic composites, modeled  by elastic energy functionals of the form
\begin{equation}\label{def:ene}
 \int_A W(\tfrac{x}\e,\nabla u(x))\,dx,\qquad u\in W^{1,p}(A,\R^d).
\end{equation}
Here, $0<\e\ll1$ stands for the period of the composite, $A\subset\R^d$ ($d\geq2$) is a bounded domain, and $W(y,F)$ denotes the stored elastic energy function which is non-convex in $F$ and which we suppose to be $Y:=[0,1)^d$ periodic in $y\in\R^d$. For the precise assumptions on $W$ see Section~\ref{S:1.2} below. In short we suppose frame indifference; minimality, non-degeneracy and smoothness at the identity; $p\geq d$-growth from below; and regularity of $W$ in the $y$-variable. We are interested in the \textit{homogenization} limit $\e\downarrow 0$. In this limit the period of the composite becomes infinitesimally small compared to the size of the macroscopic domain $A$. In the seminal works of Braides~\cite{Braides85} and M\"uller~\cite{Mueller87} it is shown that the \textit{homogenized} integral functional associated with \eqref{def:ene} is given by $\int_A W_\ho(\nabla u_0(x))\,dx$ where $W_\ho$ is defined via a \textit{multi-cell homogenization formula}
\begin{equation*}
 W_\ho(F)=\inf_{k\in\N}W_\ho\super{k}(F)\quad\mbox{where}\quad W_\ho\super{k}(F):=\inf_{\phi\in W_\per^{1,p}(kY)}\fint_{kY}W(y,F+\nabla \phi)\,dy.
\end{equation*}
The results of Braides~\cite{Braides85} and M\"uller~\cite{Mueller87} are phrased in the language of $\Gamma$-convergence of the associated energy functionals and require $W$ to satisfy standard $p$-growth (with $1<p<\infty$), see Remark~\ref{R:no_growth}. The derived homogenized energy functional $\int_A W_{\ho}(\nabla u)$ is of limited use in practice, since the evaluation of $W_{\hom}$ invokes a non-convex minimization problem on an infinite domain. The latter is also a source of difficulties regarding analytic properties of $W_{\hom}$: E.g. if $W$ satisfies standard growth-conditions, $W_{\ho}$ turns out to be quasiconvex; yet, it is in general not clear weather $W_{\ho}$ is more regular than continuous (even in the case of a two phase composite of smooth constituents).

The situation crucially simplifies when $F\mapsto W(y,F)$ is convex (e.g. for linear elasticity). In that case the multi-cell formula reduces to a single-cell formula, i.e.~$W_\ho=W_\ho^{(1)}$, see \cite{Marcellini78,Mueller87}, and in addition $W_{\hom}(F)$ can be expressed with help of a corrector, i.e. 
\begin{equation*}
  W_\ho(F)=\int_{Y}W(y,F+\nabla\phi(F))\,dy,
\end{equation*}
where $\phi(\cdot, F)\in W^{1,p}_\per(Y)$ (the \textit{corrector}) is defined as the minimizer of the minimization problem in the definition of $W_\ho\super 1(F)$. The corrector $\phi(F)$ captures the fluctuations in the strain induced by the material's heterogeneity. In general, one cannot expect that a similar simplification is valid in the non-convex case: The example of \cite{Mueller87} shows that the inequality $W_\ho\leq W_\ho^{(1)}$ (and even $W_\ho\leq W_\ho\super k$ for any finite  $k\in\N$) can be \textit{strict} if convexity of $W(y,\cdot)$ is dropped (or replaced by the weaker assumption of poly-convexity, cf.~\cite{Ball77}). Mechanically, this is related to buckling of the composites microstructure and the formation of shear bands, see \cite{Mueller93}. More recently, it is shown in \cite{BG11}, that even the inequality $W_\ho\leq \mathcal QW_\ho^{(1)}$ can be strict, where $\mathcal QW_\ho^{(1)}$ denotes the quasi-convex envelope of $W_\ho^{(1)}$. 
\medskip

Hence, it is necessary to evaluate the multi-cell formula $W_\ho$ in general. In contrast, our first result shows that (despite non-convexity) the multi-cell formula reduces to a single-cell formula for small (but finite) strains. More precisely, we show that there exists an open neighborhood $U$ of the set of rotations $\SO d$, such that for every $F\in U$ there exists a \textit{corrector} $\phi(F)\in W_\per^{1,p}(Y)$, unique up to a constant, with the property
\begin{equation*}
 W_\ho(F)=W_\ho^{(1)}(F)=\int_Y W(y,F+\nabla \phi(y,F))\,dy,
\end{equation*}
cf.~Theorem~\ref{T:1cell} and Proposition~\ref{P:layer}. As we outline next, the validity of the single-cell formula and the existence of correctors in a neighborhood of the rotations are crucial to establish further qualitative properties of $W_\ho$ (close to $\SO d$), and for deriving estimates on the homogenization error in regimes of small strains.

We begin with a discussion of the qualitative properties of $W_\ho$. In this context, we investigate the commutability issue between homogenization and linearization for nonlinear periodic composites in the spirit of Geymonat, M\"uller and Triantafyllidis \cite{Mueller93}. In \cite{Mueller93}, the authors provide, among other things, for non-convex $W$ a Taylor expansion of $W_{\hom}$ at matrices $F_0$ under the strong implicit hypothesis that (i) $W_\ho=W_\ho^{(1)}$ in a neighborhood of $F_0$ and (ii) existence and Lipschitz regularity of corresponding correctors. The question weather these properties can be deduced from ``rigorous mathematical results'' is left open, see the discussion in \cite[Section 5.2]{Mueller93}. In the present paper, we prove that both assumptions are justified for deformations close to $\SO d$ and rigorously justify the validity of the expansion in an open neighborhood of $\SO d$.  Our result can be interpreted as the \textit{commutability of homogenization and linearization} in a neighborhood of $\SO d$, see Remark~\ref{rem:commutability}. This extends earlier results by M\"uller and the first author, see \cite{MN11}, where the commutability property at the identity, i.e. for $F_0=\Id$, has been established under the mere assumption that $F\mapsto W(y,F)$ has a single non-degenerate minima at $\SO d$. For an extension of this result to the stochastic case, see \cite{GN11}, and to the multi-well case, see \cite{JS14}. Let us point out that a key issue in \cite{Mueller93} is to show that the homogenized energy density might loose strong rank-one convexity. Our result shows that this does not happen close to $\SO d$, see Theorem~\ref{T:1cell} (d). For more recent results concerning loss or conservation of strong rank-one convexity in linear elasticity we refer to \cite{BF15,FG16}. 
\medskip

A second consequence of the validity of the single-cell formula are estimates on the homogenization error. Consider the following variational problem
\begin{equation}\label{intro:ene:eps}
 \mbox{minimize}\quad \mathcal I_\e(u):=\int_A W(\tfrac{x}\e,\nabla u)-f\cdot u\,dx\qquad\mbox{subject to}\quad u-g\in W_0^{1,p}(A).
\end{equation}
The homogenization results of Braides and M\"uller \cite{Braides85, Mueller87} imply that under suitable growth conditions (almost) minimizers of \eqref{intro:ene:eps} converge to minimizers of  
\begin{equation*}
 \mbox{minimize}\quad \mathcal I_\ho(u):=\int_A W_{\hom}(\nabla u)-f\cdot u\,dx\qquad\mbox{subject to}\quad u-g\in W_0^{1,p}(A).
\end{equation*}
This result is purely qualitative and does not give any convergence rate for the minimizers. Of course, in general a rate cannot be expected, and moreover, minimizers of $\mathcal I_\e$ might not be unique or even do not exist. Assuming that the data is small in the sense that $\|f\|_{L^r(A)}+\|g-\id\|_{W^{2,r}(A)}$ for some $r>d$, we prove an error estimate in the following form of a quantitative two-scale expansion:

\begin{equation*}
 \|u_\e-u_0\|_{L^2(A)}+\|u_\e-(u_0+\e\phi(\tfrac{\cdot}\e,\nabla u))\|_{H^1(A)}\lesssim \sqrt{\e}+\left(\mathcal I_\e(u_\e)-\inf_{g+W_0^{1,p}(A)}\mathcal I_\e\right)^\frac12, 
\end{equation*}
see Theorem~\ref{T:NC:intro}, \ref{T:NC}, and Proposition~\ref{P:layer} below. To the best of our knowledge, this is the first quantitative estimate of the homogenization error in a non-convex vectorial situation. As a side remark note that our assumptions on the energy densities $W$ are not covered by the usual $p$-growth conditions (and also not by the more general convex growth conditions considered in \cite{AM11,DG16}). Hence, the known $\Gamma$-convergence result on non-convex integral functionals do not apply to our situation, and thus even the qualitative convergence result of almost minimizers of $\mathcal I_\e$ in $L^2(A)$ -- a statement that directly follows from our estimate -- seems not to be covered in the literature.

\subsection{Assumptions and main results}\label{S:1.2}
In this section we state the assumptions on $W$ and present our main results. We first introduce a class of frame-indifferent stored energy functions that are minimized, non-degenerate and smooth at identity, and satisfy a growth condition from below.
\begin{definition}\label{def:walphap}
  For $\alpha>0$ and $p>1$, we denote by $\mathcal W_{\alpha}^p$ the class of Borel functions $W:\R^{d\times d}\to[0,+\infty]$ which satisfies the following  properties:
  \begin{itemize}
  \item[(W1)] $W$ satisfies $p$-growth from below, i.e.
    \begin{equation}\label{ass:pgrowth}
      \alpha|F|^p-\frac{1}{\alpha}\leq W(F)\quad\mbox{for all $F\in\R^{d\times d}$};
    \end{equation}
  \item[(W2)] $W$ is frame indifferent, i.e.
    \begin{equation*}
      W(RF)=W(F)\quad\mbox{for all $R\in\SO d$, $F\in\R^{d\times d}$};
    \end{equation*}
  \item[(W3)] $F=\Id$ is a natural state and $W$ is non-degenerate, i.e.\ $W(\Id)=0$ and 
    \begin{align}
      W(F)&\geq \alpha\dist^2(F,\SO d)\quad\mbox{for all $F\in\R^{d\times d}$;}\label{ass:onewell}
    \end{align}
  \item[(W4)] $W$ is $C^3$ in a neighborhood of $\SO d$ and
    \begin{equation}\label{ass:wd2lip}
      \|W\|_{C^3(\overline {U_{\alpha}})}<\frac{1}{\alpha},
    \end{equation}
    where we use the notation 
    \begin{equation}\label{def:Ur}
      U_r:=\{F\in \R^{d\times d}\,:\,\dist(F,\SO d)<r\}.
    \end{equation}
  \end{itemize}
 \end{definition}
 \begin{remark}
   Note that we use the same constant $\alpha$ in the growth condition
   (W1), the non-degeneracy condition (W3) and the regularity
   assumption (W4). The only reason for this is that we want to reduce the number
   of parameters invoked in the assumption for $W$. Let us anticipate
   that the region, in which the multi-cell formula reduces to a
   single-cell expression, will depend (in a quite implicit way) on
   the constants in (W1), (W3) and (W4). Hence, working with the
   single parameter $\alpha$ simplifies the presentation. Note that we
   have $W^p_{\alpha}\subset W^{p}_{\alpha'}$ for $0<\alpha'<\alpha$.
 \end{remark}
\begin{assumption}\label{ass:W:1}
  Fix $\alpha>0$ and $p\geq d$. We suppose that
  \begin{enumerate}[(i)]
  \item  $W:\R^d\times \R^{d\times d}\to [0,+\infty]$ is a Borel function that is $Y:=[0,1)^d$ periodic in its first variable and $W(y,\cdot)\in\mathcal W_{\alpha}^p$ for almost every $y\in Y$.
  \item  $W\in C^3(\R^d\times \overline {U_\alpha})$. 
  \end{enumerate}
\end{assumption}
We will see below, that the smoothness assumption (ii) can be replaced by the assumption that the composite is layered, e.g. $W(y,F)=W(y_1,F)$, see Assumption~\ref{ass:W:layerd}.

\paragraph{Validity of the single-cell homogenization formula.} Our first result proves the validity of the single-cell formula for small, but finite strains:
\begin{theorem}\label{T:1cell}
Suppose Assumption~\ref{ass:W:1} is satisfied. Then there exists $\bar{\bar\rho}>0$  such that for all 
\begin{equation*}
  F\in U_{\bar{\bar\rho}}:=\big\{\,F\in \R^{d\times d}\,:\,\dist(F,\SO d)<\bar{\bar\rho}\,\big\}
\end{equation*}
the following properties are satisfies:
\begin{enumerate}[(a)]
\item (Single-cell formula). 
  \begin{equation*}
    W_\ho(F)=W_\ho\super{1}(F).
  \end{equation*}
\item (Corrector). There exists a unique corrector $\phi(F)\in W^{1,p}_{\per}(Y,\R^d)$ such that $\int_Y\phi(F)=0$ and
  \begin{eqnarray*}
      W_\ho(F)&=&\int_YW(y,F+\nabla \phi(y,F))\,dy.
  \end{eqnarray*}
  The corrector satisfies $\phi(F)\in W^{1,\infty}(Y)$.
\item (Regularity and quadratic expansion). $W_\ho\in C^3(U_{\bar{\bar\rho}})$ and for all $G\in\R^{d\times d}$ we have 
\begin{align}\label{C:d2Whom:claim}
 DW_\ho(F)[G]=&\int_YDW(y,F+\nabla \phi(y,F))[G]\,dy\notag,\\
 D^2W_\ho(F)[G,G]=&\inf_{\psi\in H_\per^1(Y)}\int_YD^2W(y,F+\nabla \phi(y,F))[G+\nabla \psi(y),G+\nabla \psi(y)]\,dy,
\end{align}
where $\phi(F)$ denotes the corrector defined in (b).
\item (Strong rank-one convexity). There exists $c>0$ such that 
\begin{equation}\label{limit:nondegenerate}
  D^2W_\ho(F)[a\otimes b,a\otimes b]\geq  c|a\otimes b|^2\quad\mbox{for all $a,b\in\R^d$}.
\end{equation}
\end{enumerate}
\end{theorem}
The proof of Theorem~\ref{T:1cell} is given in Section~\ref{sec:1cell}. As indicated in the introduction, the key part of Theorem~\ref{T:1cell} is the validity of the single-cell formula and the existence of correctors. Next, we give a short overview of the proof. A key ingredient is the construction of a \textit{matching convex lower bound}. It is based on the following observation: For every $W\in\mathcal W_p^\alpha$, $p\geq d$, there exists a perturbation by a Null-Lagrangian $\overline W:=W+\mu\det$, for some $\mu>0$, such that $\overline W$ is strongly convex in a neighborhood of $\SO d$. It is then possible to construct a \textit{strongly convex} function $V$ satisfying $\overline W\geq V$ on $\R^{d\times d}$ and $\overline W(F)=V(F)$ for all $F$ close to $\SO d$ (more precisely, for all $F\in U_\delta$ with $\delta>0$ only depending on $\alpha, p, d$ and $\mu$, and $U_\delta$ defined as in  \eqref{def:Ur}). 
\smallskip

This crucial observation and the construction of $V$ has been established in a discrete setting in the works by Friesecke and Theil~\cite{FT02} and Conti et.~al.~\cite{CDMK06}, where the validity of the Cauchy-Born rule for discrete elastic energies is studied. More precisely, they prove that minimizers of certain discrete energies subject to an affine boundary condition (assumed to be close to a rigid motion), are affine. In Lemma~\ref{L:wv} below we adapt their construction to the continuum setting and obtain for any $W\in \mathcal W_p^\alpha$ a convex energy density $V$ satisfying the following properties:
\begin{enumerate}[(a)]
 \item For all cubes $Q\subset \R^d$,  $\varphi\in W_0^{1,\infty}(Q)$, and $F\in \R^{d\times d}$ we have
$$\fint_QW(F+\nabla \varphi)\geq\fint_QV(F+\nabla \varphi)-\mu\det(F+\nabla \varphi)\geq V(F)-\mu\det(F),$$
\item For all $F\in U_\delta$ we have $W(F)=V(F)-\mu\det(F)$.
\end{enumerate}
We then exploit this construction in the context of homogenization: To any periodic energy density $W(y,F)$ satisfying Assumption~\ref{ass:W:1} (i), we can associate a \textit{matching convex lower bound} $V=V(y,F)$, which satisfies the following variant of  (a):
\begin{enumerate}[(a)]
\item[(a')] For all $F\in \R^{d\times d}$ we have
  $$W_\ho(F)\geq V_\ho^{(1)}(F)-\mu\det(F).$$
\end{enumerate}
In the derivation of (a') we exploit the convexity of $V$ in form of the identity $V_\ho=V_\ho^{(1)}$, and appeal to the fact that $\det$ (as a Null-Lagrangian) is invariant w.r.t. periodic fluctuations. To conclude $W_\ho=W_\ho\super1$ (close to $\SO d$), we require also a variant of the matching property (b), namely: $W_\ho^{(1)}(F)=V_\ho^{(1)}(F)-\mu\det(F)$ for all $F$ close to $\SO d$. We obtain this identity by showing that the corrector $\phi(F)$ associated with the convex energy density $V$ is small in the sense of $\|F+\nabla \phi(F)\|_{L^\infty(Y)}<\delta$. We achieve this by appealing to the implicit function theorem, which we apply to the corresponding Euler-Lagrange equation $-\divv DV(y,F+\nabla \phi(y))=0$. As a consequence we obtain a representation of $W_\ho(F)$ in terms of the corrector $\phi(F)$, which turns out to be the unique minimizer of the minimization problem in the definition of $W_\ho\super 1(F)$.
We would like to remark that although the definition of the corrector $\phi(F)$ is purely variational (it is a minimizer of a convex integral functional), for the Lipschitz estimate, it seems to be necessary to use non-variational regularity techniques such as the implicit function theorem. This is also the only place where we require the additional smoothness assumption \eqref{ass:W:1} (ii).  In dimension $d=2$, we might replace the non-variational regularity argument by appealing to direct methods for regularity  of smooth convex integral functionals.

\begin{remark}\label{rem:commutability}
Let us briefly comment on the commutability of homogenization and linearization in a neighborhood of the rotations. In the situation of Theorem~\ref{T:1cell} let $F\in U_{\bar{\bar\rho}}$ and consider the deformation
\begin{equation*}
  u_\e(x):=Fx+\e\phi(\tfrac{x}{\e},F).
\end{equation*}
Part (a) and (b) of the theorem imply that $u_\e$ is an equilibrium state in the sense that
\begin{equation*}
  \int_{\R^d} DW(\tfrac{x}{\e},\nabla u_\e(x))[\nabla\eta(x)]\,dx=0\qquad\text{for all }\eta\in C^\infty_c(\R^d).
\end{equation*}
Moreover, on a macroscopic scale $u_\e$ behaves as a linear deformation with slope $F$. Indeed, we have $u_\e \to u_0(x):=Fx$ strongly in $L^p_{\loc}(\R^d)$. We argue that homogenization and linearization at $u_\e$ (resp. $u_0$) commute in the following sense: Consider the functional 
\begin{equation*}
 \mathcal G_{\e,h}\,:\,W^{1,p}_0(A,\R^d)\to\R,\qquad \mathcal G_{\e,h}(\varphi):=\frac{1}{h^2}\int_AW(\tfrac{x}{\e},\nabla u_\e(x)+h\nabla \varphi(x))-W(\tfrac{x}\e,\nabla u_\e(x))\,dx,
\end{equation*}
and suppose that $W$ satisfies standard $p$-growth. Then the homogenization result for non-convex integral functionals (e.g. \cite{Mueller87}) implies that as $\e\downarrow 0$, $\mathcal G_{\e,h}$ $\Gamma$-converges (in $L^p$) to 
\begin{equation*}
  \mathcal G_{\ho,h}(\varphi):=\frac{1}{h^2}\int_A W_{\hom}(F+h\nabla\varphi(x))-W_\ho(F)\,dx.
\end{equation*}
Furthermore, thanks to the $C^3$-regularity of $W_{\ho}$ (see statement (c)), we have the (pointwise) linearization limit
\begin{equation*}
 \mathcal G_{\ho,\operatorname{lin}}(\varphi):=\lim\limits_{h\downarrow 0}\mathcal G_{\hom,h}(\varphi)=\int_A \tfrac12 D^2W_{\hom}(F)[\nabla\varphi(x),\nabla\varphi(x)]\,dx.
\end{equation*}
On the other hand, by first linearizing we get 
\begin{equation*}
 \mathcal G_{\e,\operatorname{lin}}(\varphi):=\lim\limits_{h\downarrow 0}\mathcal G_{\e,h}(\varphi)=\int_A \tfrac12 D^2W(\tfrac{x}{\e},\nabla u_\e(x))[\nabla\varphi(x),\nabla\varphi(x)]\,dx,
\end{equation*}
which holds thanks to the Lipschitz estimate on $\phi(F)$ and the regularity of $W$ close to $\SO d$. Since $D^2W(\frac{x}{\e},\nabla u_\e(x))=D^2W(y,F+\nabla\phi(y,F))\vert_{y=\frac{x}{\e}}$ is a (rescaled), periodic second order tensor which is (thanks to $W(\cdot,\nabla u_\e)=V(\cdot,\nabla u_\e)-\mu\det(\nabla u_\e)$, see Corollary~\ref{C:corrector}) elliptic in the integrated form, linear elliptic homogenization and formula \eqref{C:d2Whom:claim} show that $\mathcal G_{\e,\operatorname{lin}}$ $\Gamma$-converges (in $L^2$) to the functional $\mathcal G_{\ho,\operatorname{lin}}$ as well. Hence, both paths lead to the same limit and we may say that homogenization and linearization at $u_\e$ (resp. $u_0$) commute. In future work, we show that the pointwise linearization limit in the above argument can be replaced by $\Gamma$-convergence and provide quantitative estimates for the difference between minimizers for certain non-linear energies and its linearization.

\end{remark}

\paragraph{Quantitative two-scale expansion and homogenization error.}%
With help of the single-cell formula and its representation with help of the corrector, i.e. Theorem~\ref{T:1cell}, we are in position to provide estimates on the homogenization error in regimes of small strains. To this end, we introduce the functionals 
\begin{align*}
  \mathcal I_{\e}(u):=\int_AW(\tfrac{x}\e,\nabla u)-f\cdot u\,dx,\quad  \mathcal I_{\hom}(u):=\int_AW_{\hom}(\nabla u)-f\cdot u\,dx.
\end{align*}
We have the following result, which to the best of our knowledge, is the first quantitative homogenization result for non-convex, vectorial integral functionals:
\begin{theorem}[Homogenization error]\label{T:NC:intro}
Suppose Assumption~\ref{ass:W:1} is satisfied. Let $A\subset\R^d$ be a bounded domain of class $C^2$, and fix $r>d$. Then there exists $\bar \rho\ll 1$  such that following statement holds:
\smallskip

If the data $f\in L^r(\R^d)$ and  $g\in W^{2,r}(A)$ is small in the sense of
\begin{equation}\label{def:Lambda:intro}
 \Lambda(f,g):=\|f\|_{L^r(A)}+\|g-\id\|_{W^{2,r}(A)}<\bar\rho,
\end{equation}
then:
\begin{enumerate}[(a)]
 \item The minimization problem $\inf\{\mathcal I_\ho( u)\, |\, u\in g+W_0^{1,p}(A)\}$ admits a unique minimizer $u_0\in W^{2,r}(A)$.
 \item For every $u_\e\in g+W_0^{1,p}(A)$, the following estimate is true
    \begin{align}\label{estimate:intro}
      \|u_\e-u_0\|_{L^2(A)}+\|u_\e-(u_0+\e\nabla\phi(\tfrac{\cdot}\e,\nabla u_0))\|_{H^1(A)}&\lesssim \e^\frac12\Lambda(f,g) +\left(\mathcal I_\e(u_\e)-\inf_{g+W_0^{1,p}(A)} \mathcal I_\e\right)^\frac12 
    \end{align}
    where $\lesssim$ means $\leq$ up to a constant only depending on $d, A, \alpha,p$ and $\bar\rho$.
  \end{enumerate}
\end{theorem}
In Section~\ref{sec:hom}, we prove a slightly more general result, see Theorem~\ref{T:NC}, where the smallness assumption on $\|g-\id\|_{W^{2,r}(A)}$ regarding the boundary data is relaxed in the sense that the identity map $\id(x)=x$ can be replaced by a sufficiently smooth equilibrium deformation. The theorem above is a special case of Theorem~\ref{T:NC}.
\smallskip

Clearly, if $\mathcal I_\e$ admits a minimizer in $g+W_0^{1,p}(A)$, Theorem~\eqref{T:NC:intro} yields an error estimate for the corresponding minimizer. The $\e^{\frac12}$ scaling for the $H^1$ error in the two-scale expansion matches the corresponding estimates for linear elliptic systems (without boundary layer corrections), see below, and we thus expect this to be optimal. 

Notice that in view of the counterexample to the single-cell formula given in \cite[Theorem 4.3]{Mueller87}, the minimizer of $\mathcal I_\e$ might oscillate on a scale much larger than the periodicity-cell, and thus the two-scale expansion is not an appropriate approximation in general. Hence, some smallness assumptions,  such as \eqref{def:Lambda:intro}, on the data are necessary for the validity of estimate \eqref{estimate:intro}. 

\medskip

Estimates on the homogenization error and the error term in two-scale expansions have a long history in homogenization theory. Classical results mainly treat the case of linear elliptic equations or systems in divergence form. In the following, we briefly review some classical and recent results. We start with the scalar problem
  $$u_\e\in H_0^1(A),\quad -\divv (a(\tfrac{\cdot}{\e})\nabla u_\e)=f\quad\mbox{in $A$},$$
  where $a:\R^d\to\R^{d\times d}_{\sym}$ is $Y$-periodic and uniformly elliptic and bounded. For $f\in L^2(A)$, the solution $u_\e$ converges weakly in $H^1(A)$ to the unique solution of
  $$u_0\in H_0^1(A),\quad -\divv(a_\ho\nabla u)=f\quad\mbox{in $A$}.$$
  In contrast to the non-linear case, for linear equations the corrector tensorizes in the sense that $\phi(\xi)=\phi_{i}\xi_{i}$ where $\phi_{i}$ is the corrector for $e_i$. Hence, the two-scale expansion reads $u_0+\e \phi_{i}\partial_i u_0$, where $u_0$ is the unique solution of the homogenized problem. If one has either additional regularity of the corrector, namely $\phi_{i}\in W^{1,\infty}(Y)$, or of the solution to the homogenized problem, namely $u_0\in W^{2,\infty}(A)$, then the quantitative two-scale expansion
$$\|u_\e-(u_0+\e \phi_{i}\partial_i u_0)\|_{H^1(A)}\leq C\e^\frac12,$$ 
holds. This is completely classical and can be found in \cite{BLP75}. The required regularity properties on $\phi$ (resp. $u_0$) can be justified e.g. by assuming $a\in C^{0,\alpha}$ (resp. $f\in C^{0,\alpha}$ with $\alpha\in(0,1)$). In view of the two-scale expansion, one might naively expect the estimate to hold with rate $\e$ (instead of $\e^{\frac12}$). Yet, this in general not true due to the formation of boundary layers ($u_\e$ and the two-scale expansion do not satisfy the same boundary conditions).  By appealing to boundary layer corrections (i.e. adding an additional term to the two-scale expansion), and the maximum principle it is also classical that
$$\|u_\e-u_0\|_{L^2(A)}\leq C\e,$$
even without assuming additional regularity assumptions on the coefficients or the right-hand side $f$, see e.g. \cite{JKO-94}. Let us mention that by now, the above $L^2$ and $H^1$ estimates are available without appealing to any extra regularity assumptions or the maximum principle, and are also valid for elliptic systems (with a regularized version of the two-scale expansion), see \cite{Griso04,ZP05,S13,SZ17}.
\smallskip

For non-linear problems there are only few results available in the literature. In \cite{CPZ05,P08}, the authors consider monotone operators in divergence form. In particular, \cite{P08} studies the (scalar) homogenization problem 
$$u_\e\in H_0^1(A),\quad-\divv a(\tfrac{x}\e,\nabla u_\e)=f\quad\mbox{in $A$,}$$
where $a$ is periodic in the first component and satisfies in the second component certain monotonicity, continuity and growth conditions. It is shown that
$$\|u_\e-u_0\|_{L^2(A)}+\|u_\e-(u_0+\e \phi(\tfrac{x}\e,\nabla u_0))\|_{H^1(A)}\leq C\e^{\frac12}\|f\|_{L^2(A)},$$
where $u_0$ is the solution of the corresponding homogenized problem and $\phi$ is the corrector. The proof of the above estimate uses the DeGiorgi-Nash-Moser regularity for nonlinear elliptic equations. In the course of proving Theorem~\ref{T:NC:intro}, we will also establish a similar estimate in the case of a monotone systems (the Euler-Lagrange equation of a strictly convex minimization problem). Since scalar regularity methods cannot be applied, we appeal to non-variational regularity arguments that require a certain smallness of the data, see Proposition~\ref{P:C}. In principle those smallness assumptions might be considered unsatisfactory for the estimate in the convex case, however in view of our main goal, namely the error estimate for the non-convex case, they seem necessary.

\smallskip

\begin{remark}\label{R:no_growth}
Since no growth conditions from above are assumed for the integrand $W$ in Theorem~\ref{T:NC:intro} the known $\Gamma$-convergence results for homogenization of integral functionals do not apply in this situation. Hence, even the qualitative statement, that is (almost) minimizer of $\inf\{\mathcal I_\e (u)\, |\, u\in g+W_0^{1,p}(A)\}$ converge to minimizer of $\inf\{\mathcal I_\ho(u)\, |\, u\in g+W_0^{1,p}(A)\}$, seems not available in the literature yet.
\end{remark}

\begin{remark}
 Appealing to the Gagliardo-Nirenberg interpolation inequality and the $p\geq d$-growth from below of $W$, we obtain that $\|u_\e-u_0\|_{L^2(A)}$ in \eqref{estimate:intro} can be replaced by $\|u_\e-u_0\|_{L^q(A)}$ for every $q\in[1,\infty)$ (and $q=\infty$ if $p>d$) by the expense that $\e^\frac12$ is replaced by $\e^\frac\delta2$ with $\delta=\delta(d,p,q)\in(0,1]$ and the estimate depends in a non-linear way on $\mathcal I_\e(u_\e)$ and $\Lambda(f,g)$.  
\end{remark}

\begin{remark}
 The proof of part (a) of Theorem~\ref{T:NC:intro} reveals that the global minimizer of $\mathcal I_\ho$ in $g+W_0^{1,p}(A)$ and the solution of the corresponding Euler-Lagrange equation obtained by a implicit function theorem coincide. This can be compared to the work of Zhang~\cite{Zhang91} (see also \cite{CD02}), where it is shown that the global energy minimizer and the solution of the corresponding Euler-Lagrange equation, obtained by the implicit function theorem, coincide if the data is sufficiently small. While in \cite{Zhang91}, the energy density is assumed to be of a special polyconvex structure, we do not assume any global convexity notion but use existence of a matching convex lower bound, see below.
\end{remark}

\subsection{Layered composite}

The smoothness assumption of the energy density in the space variable, cf.~Assumption~\ref{ass:W:1} (ii), can be omitted by considering only variations along one direction:
\begin{assumption}\label{ass:W:layerd}
  Fix $\alpha>0$, $p\geq d$ and $N\in\N$. Let $0=t_0<t_1<\dots<t_N=1$ and $\{W_i\}_{i=1,\dots,N}$ with $W_i\in W_\alpha^p$ for $i=1,\dots,N$ be given. Consider $W :\R^d\times \R^{d\times d}\to [0,+\infty]$ which is $Y$ periodic in the first variable and satisfies for almost every $y\in Y$ and every $F\in \R^{d\times d}$ that
  \begin{equation*}
   W(y,F)=\sum_{i=1}^N \chi_i(y_1)W_i(F)\quad\mbox{where }\chi_i(t)=\begin{cases}
                                                                     1&\mbox{if $t\in(t_{i-1},t_i)$},\\
                                                                     0&\mbox{else.}
                                                                    \end{cases}
  \end{equation*}
\end{assumption}

\begin{proposition}\label{P:layer}
  The conclusion of Theorem~\ref{T:1cell} and Theorem~\ref{T:NC} remain valid if Assumption~\ref{ass:W:1} is replaced by  Assumption~\ref{ass:W:layerd}.
\end{proposition}

\begin{remark}
 Consider $W$ satisfying Assumption~\ref{ass:W:layerd}. The example of M\"uller, \cite[Theorem 4.3]{Mueller87}, shows $W_\ho(F_\lambda)<W_\ho^{(1)}(F_\lambda)$, where $F_\lambda=\operatorname{diag}(\lambda,1)$ for some $\lambda<1$. In fact, $\lambda<1$ can be chosen arbitrary close to $1$ provided $\alpha>0$ in Assumption~\ref{ass:W:layerd} is sufficiently small, hence the neighborhood of the rotations where the single-cell formula is valid vanishes when the energy density becomes degenerate. 
\end{remark}

\subsection{Organization of the paper}

In Section~\ref{sec:V}, we give a detailed proof for the construction of the convex lower bound. 
Section~\ref{sec:corrector} is devoted to the (extended) corrector. In particular, we establish the regularity results required in the proofs of our main results. In Section~\ref{sec:1cell}, we prove Theorem~\ref{T:1cell} and in Section~\ref{sec:hom} we provide a proof of the error estimate. Finally, in Section~\ref{sec:layer} we treat the case of layered composites.

\subsection*{Notation.}
\begin{itemize}
 \item $d\geq2$ dimension of the domain,
 \item $f\cdot g$ (resp. $\langle F,G\rangle$) denotes the standard scalar product in $\R^d$ (resp. $\R^{d\times d}$),   
 \item $Y=[0,1)^d$ reference cell of periodicity,
 \item For $r>0$ set $U_r:=\{F\in\R^{d\times d}:\dist(F,\SO d)<r\}$,
 \item For $f:\R^d\times \R^{d\times d}\to\R^n$ with $n\in\N$, we denote
 \begin{itemize}
  \item by $\nabla f$ the Jacobian matrix, i.e. the derivative with respect to the spatial variable: $(\nabla f(x,F))_{ij}=\partial_{x_j}f_i(x,F)=\partial_jf_i(x,F)$.
  \item by  $D^kf$ the $k$-th Fr\'echet-derivative with respect to the second component. If $n=1$, we identify the linear (resp. bilinear) map $Df(\cdot,F)$ (resp. $D^2f(\cdot, F)$) with the matrix (resp. fourth order tensor) defined by $Df(\cdot,F)[G]=\langle Df(\cdot,F),G\rangle$ (resp. $D^2f(\cdot,F)[G,H]=\langle D^2f(\cdot,F)G,H\rangle$).
\end{itemize}
 \item A measurable function $u:\R^d\to\R$ is called $Y$-periodic, if it satisfies $u(y+z)=u(y)$ for almost every $y\in\R^d$ and all $z\in\Z^d$. We consider the function spaces
 \begin{align*}
  L_\per^p(Y)&:=\{u\in L_{\loc}^p(\R^d)\, |\, u\mbox{ is $Y$-periodic}\},\quad L_{\per,0}^p(Y):=\{u\in L_\per^p(Y)\, |\, \int_Y u=0\},\\
  W_\per^{1,p}(Y)&:=\{u\in L_{\per}^p(Y)\, |\, u\in W_\loc^{1,p}(\R^d)\},\quad  W_{\per,0}^{1,p}(Y):=\{u\in W_\per^{1,p}(Y)\, |\, \int_Y u=0\},
 \end{align*}
and likewise $L_\per^p(Y,\R^d)$, $L_{\per,0}^p(Y,\R^d)$ etc.. Throughout the paper, we drop the explicit dependence of the target space whenever it is clear from the context.

 \end{itemize}

\section{Matching convex lower bound} \label{sec:V}
In this section we construct the matching convex lower bound associated with non-convex potentials of class $\mathcal W_\alpha^p$. It turns out to be a strongly convex potential belonging to the following class:
\begin{definition}
  For $\beta>0$ we denote by $\mathcal V_{\beta}$ the set of functions $V\in C^2(\R^{d\times d})$ satisfying
   \begin{align}
   &\beta |F|^2-\frac{1}{\beta}\leq V(F)\leq \frac{1}{\beta}(|F|^2+1),\label{vgrowth}\\
      & |DV(F)[G]|\leq \frac1\beta(1+|F|)|G|,\label{vgrowthd1}\\
   &\beta|G|^2\leq D^2V(F)[G,G]\leq \frac{1}{\beta}|G|^2. \label{vgrowthd2}
 \end{align}
\end{definition}
\begin{remark}\label{R:strong-convexity}
At several places in this paper we appeal to strong convexity of $V\in\mathcal V_{\beta}$ in form of the inequalities
\begin{equation}\label{ineq:strong-conv}
 \forall F,G\in\R^{d\times d}\,:\qquad \beta|F-G|^2\leq \big(DV(F)-DV(G)\big)[F-G],
\end{equation}
and
\begin{equation}\label{ineq:strong-conv2}
 \forall F,G\in\R^{d\times d}\,:\qquad \frac{\beta}{2}|F-G|^2\leq V(G)-(V(F)+DV(F)[G-F]).
\end{equation}
\end{remark}

Next, we state the existence of a matching convex lower bound for all $W\in W_\alpha^p$ with $p\geq d$:
\begin{lemma}\label{L:wv}
 Consider $W\in \mathcal W_{\alpha}^p$ with $p\geq d$. Then there exists $\bar \mu>0$ such that for any $0<\mu\leq\bar\mu$ we can find $\delta,\beta>0$, and a (strongly convex) function $V\in \mathcal V_{\beta}$ such that
  \begin{align}
    W(F)+\mu \det F&\geq V(F)\qquad\mbox{for all $F\in \R^{d\times d}$,}\label{WgeqV0}\\
    W(F)+\mu \det F&= V(F)\qquad \mbox{for all }F\in  U_\delta,\label{W=V0}
  \end{align}
  and
  \begin{equation}\label{V:frame}
    V\in C^3(\R^{d\times d})\quad\mbox{and}\quad V(RF)=V(F)\qquad\text{for all }F\in \R^{d\times d}, R\in\SO d.
  \end{equation}
\end{lemma}
A direct consequence of Lemma~\ref{L:wv} and Assumption~\eqref{ass:W:1} (i) is the following:
\begin{corollary}\label{C:wv}
  Suppose that Assumption~\eqref{ass:W:1} (i) is satisfied. Then there exist $\delta,\mu,\beta>0$, and a Caratheodory function $V:\R^d\times\R^{d\times d}\to \R$ which is $Y$-periodic in its first variable and $V(y,\cdot)\in \mathcal V_{\beta}$ satisfies \eqref{V:frame} for almost every $y\in Y$ such that
\begin{align}
 W(y,F)+\mu \det F&\geq V (y,F)\qquad\mbox{for almost every $y\in Y$ and all $F\in \R^{d\times d}$,}\label{WgeqV}\\
 W(y,F)+\mu \det F&=V (y,F)\qquad \mbox{for almost every $y\in Y$ and all }F\in  U_\delta.\label{W=V}
\end{align}
\end{corollary}
The rest of this section is devoted to the proof of Lemma~\ref{L:wv}. The short argument for Corollary~\ref{C:wv} is left to the reader.

\begin{proof}[Proof of Lemma~\ref{L:wv}]
The proof relies on a convexification of the modified potential
$$\ol W(F):=W(F)+\mu\det F\qquad (\mu>0),$$
which (in contrast to $W$) is locally strongly convex at $\SO d$ (i.e. $\ol W$ restricted to a small ball around a rotation is strongly convex). In Step~1 and Step~2, we present the convexification argument, following the argument in \cite[Theorem 4.2]{CDMK06}, where the analogous result is obtained in a discrete setting. Since the discrete case contains additional technicalities, we present the continuum argument for the reader's convenience. The convex potential obtained in Step~2 already satisfies all the claimed properties, except for the asserted smoothness properties. The latter is obtained in  Step 3 by a careful regularization.

\step{1 - Local strong convexity of $\ol W$}%
We claim that for all $0<\mu\ll 1$ (where here and below $\ll$ means $\leq$ up to constant only depending on $d$, $p$ and $\alpha$ from \eqref{ass:onewell}) there exists  $\delta=\delta(d,\alpha,\mu)\in(0,1)$ and constant $C=C(d)>0$ such that
\begin{align}\label{est:g:1}
  &\ol W(F)-\ol W(F_0)- D \ol W(F_0)[F-F_0]\geq \frac{\mu}{C} |F-F_0|^2\\\notag
  &\mbox{for all $F,F_0\in\R^{d\times d}$ with $\dist(F_0,\SO d)\leq \delta$.}
\end{align}
The argument is divided in three steps.

\substep{1.1 -  Strict Legendre ellipticity close to $\SO d$}

We claim that there exists a constant $C=C(d)$  such that whenever $0\leq\mu\leq \frac{\alpha}{C}$ we have for some $\delta'=\delta'(d,\alpha,\mu)>0$:
\begin{equation*}
  D^2\ol W(F_0)[G,G]\geq\frac{\mu}{2C}|G|^2\qquad\text{for all }G,F_0\in\R^{d\times d}\text{ with }\dist(F_0,\SO d)\leq 2\delta'.
\end{equation*}
Since the quadratic forms $D^2\det$ and $D^2W$ are continuous in a neighborhood of $\SO d$, see \eqref{ass:wd2lip}, it suffices to prove 
\begin{equation}\label{est:ellbarW}
  D^2\ol W(R)[G,G]\geq\frac{\mu}{C}|G|^2\qquad\text{for all }G\in\R^{d\times d}\text{ and }R\in\SO d.
\end{equation}
By frame-indifference we may assume without loss of generality that $R=I$. The argument relies on the following estimate which is proven in \cite[Proof of Theorem~4.2, Step~1]{CDMK06}
\begin{align}\label{est:det1}
\forall F,F'\in\R^{d\times d}\,:\qquad
  &\det F-\det F'- D\det(F')[F-F']\geq \frac{1}{C}|F-F'|^2\notag\\
  &\hspace{3cm}-C(\dist^2(F,\SO d)+\dist^d(F,\SO d))\notag\\
  &\hspace{3cm}-C(\dist^2(F',\SO d)+\dist^d(F',\SO d)),
\end{align}
where  $C=C(d)$. By linearization at identity we obtain $D^2 \det(I)[G,G] \geq \frac{1}{C}|G|^2-C|\sym G|^2$.
Thanks to the non-degeneracy of $W$, see \eqref{ass:onewell}, we have $D^2 W(I)[G,G]\geq \alpha|\sym G|^2$, and thus \eqref{est:ellbarW} holds provided $\mu\leq \frac{\alpha}{C}$.

\substep{1.2 - Proof of \eqref{est:g:1}}%

Let $\mu$, $C$ and $\delta'$ be as above, and suppose that $\dist(F_0,\SO d)<\delta$ where $0<\delta<\delta'$ will be fixed below. By frame-indifference we may assume without loss of generality that $\dist(F_0,\SO d)=|F_0-I|$. If $|F-F_0|<\delta'$, then a Taylor expansion yields
\begin{equation*}
  [\text{LHS of }\eqref{est:g:1}]= \int_0^1(1-s)D^2\ol W(F_0+s(F-F_0))[F-F_0,F-F_0]\,ds.
\end{equation*}
Since $|F_0+s(F-F_0)-I|<2\delta'$ we may apply Substep 1.1 and get
\begin{equation*}
  [\text{LHS of }\eqref{est:g:1}]\geq \frac{\mu}{4C}|F-F_0|^2.
\end{equation*}
If $|F-F_0|\geq \delta'$ we proceed as follows:
Due to the growth condition \eqref{ass:pgrowth} (where $p\geq d$) and \eqref{ass:onewell}, we can find $C'=C'(d,\alpha,p)>0$ such that
\begin{equation*}
  W(F)\geq C'\left(\dist^2(F,\SO d)+\dist^d(F,\SO d)\right).
\end{equation*}
Combined with \eqref{est:det1} we get for $0<\mu\leq \min\{\frac{\alpha}{C},\frac{C'}{C}\}$:
\begin{eqnarray}\notag
  [\text{LHS of }\eqref{est:g:1}]
  &=&\ol W(F)-\mu\big(\det(F_0)+ D\det (F_0) [F-F_0]\big)\\\notag
  &&-\left(W(F_0)+DW(F_0)[F-F_0]\right)\\\notag
  &\geq& \frac{\mu}{C}|F-F_0|^2+(C'-\mu C)\left(\dist^2(F,\SO d)+\dist^d(F,\SO d)\right)\\\notag
  &&-2\mu C\delta^2-\left(W(F_0)+DW(F_0)[F-F_0]\right)\\\notag
  &\geq& \frac{\mu}{4C}|F-F_0|^2+\left(\frac{3\mu}{4C}\delta'^2-2\mu C\delta^2-\left(W(F_0)+DW(F_0)[F-F_0]\right)\right),
\end{eqnarray}
where we use that $|F-F_0|^2\geq \frac14|F-F_0|^2+\frac34{\delta'}^2$, $C'-\mu C\geq 0$, and $\dist^2(F_0,\SO d)+\dist^d(F_0,\SO d)\leq 2\delta^2$. We finally conclude by arguing that we may choose $\delta=\delta(d,\alpha,\mu,\delta')$ sufficiently small such that the second term on the right-hand side is non-negative. This is indeed possible, since
\begin{eqnarray*}
  |W(F_0)|&\leq&|\int_0^1(1-s)D^2W(I+s(F_0-I))[F_0-I,F_0-I]|\leq C''\delta^2,\\
  |DW(F_0)[F-F_0]|&\leq&|\int_0^1D^2W(I+s(F_0-I))[F-F_0,F_0-I]|\leq C''|F-F_0|\delta\leq C''\delta'\delta,
\end{eqnarray*}
for $C''=C''(d,\alpha)$, as follows from a  Taylor expansion and the fact that $W$ is minimized at identity.

\step{2 -- Convexification}

For the rest of the proof we fix $C,\delta,\mu>0$ according to Step~1. For notational convenience, we set $\lambda:=\frac{\mu}{2C}$ and write $\lesssim$ if the relation holds up to a constant only depending on $d,\alpha,\mu,\delta$ and $C$. We convexify $\ol W$ as follows: For $F_0\in U_{\delta}$ we define a quadratic form $\ol Q_{F_0}(\cdot)$ as
\begin{equation*}
  \ol Q_{F_0}(F):=\ol W(F_0)+D\ol W(F_0)[F-F_0]+\lambda|F-F_0|^2,
\end{equation*}
and a potential $\ol V(\cdot)$ by
\begin{equation*}
  \ol V(F):=\sup_{F_0\in U_\delta}\ol Q_{F_0}(F).
\end{equation*}
We note that:
\begin{enumerate}[(a)]
\item $\ol V$ is strongly $2\lambda$-convex, that is
\begin{equation}\label{est:strong-conv}
  \ol V(\cdot)-\lambda|\cdot|^2\text{ is convex.}
\end{equation}
Indeed, by construction the quadratic potentials $\ol Q_{F_0}$ are strongly $2\lambda$-convex, and thus this property is inherited by the supremum.
\item $\ol V$ is a lower bound for $\ol W$ in the sense that
\begin{equation}
\label{est:comparison}
  \begin{aligned}
    \forall F\in\R^{d\times d}\,:&\qquad \ol W(F)\geq\ol V(F)+\lambda\dist^2(F,U_\delta),\\
    \forall F\in U_\delta\,:&\qquad \ol W(F)=\ol V(F).
  \end{aligned}
\end{equation}
Indeed, the inequality is a direct consequence of \eqref{est:g:1} (recall that $\lambda=\frac{\mu}{2C}$), and the identity holds, since $\ol Q_F(F)=\ol W(F)$ for $F\in U_\delta$ by construction.
\item $\ol V$ is frame-indifferent and satisfies the quadratic growth conditions:
\begin{equation}\label{est:growth}
  \lambda|F|^2\leq \ol V(F)-\min\ol V\leq \tfrac43\lambda|F|^2+C'.
\end{equation}
for some constant $C'\lesssim 1$. We first notice that $\min\ol V=\ol V(0)$. Indeed, by strong convexity $\ol V$ has a unique minimizer, say $F_*$. The frame-indifference of $\ol W$ implies that $\ol V$ is frame-indifferent, and thus $F_*=RF_*$ for all $R\in\SO d$. This implies $F_*=0$. The lower bound is a direct consequence of $\min\ol V=\ol V(0)$ and \eqref{est:strong-conv}. Moreover, we have for all $F_0\in U_\delta$ 
\begin{equation*}
  \ol Q_{F_0}(F)-\min\ol V=  \ol Q_{F_0}(F)-\ol V(0)\leq \ol Q_{F_0}(F)-\ol Q_{F_0}(0)=D\ol W(F_0)[F]+\lambda(|F|^2-2F\cdot F_0),
\end{equation*}
and thus the claimed upper bound follows by Young's inequality.
\end{enumerate}
Let us remark that up to this stage we basically adapted the argument of \cite[Theorem 4.2]{CDMK06} to the continuum setting and the constructed potential $\ol V$ satisfies all the sought after properties except for the $C^3$-regularity. In order to obtain the latter we carefully regularize $\ol V$ in the next step.

\step{3 -- Regularization}

We first claim that there exists a positive constant $0\leq C_0\lesssim 1$ and a radius $R\lesssim 1$ with $U_\delta\subset B_R:=\{F\in\R^{d\times d}\,|\,|F|\leq R\}$, such that
\begin{equation*}
  \widehat V(F):=\max\{\ol V(F),Q(F)\}\qquad\text{ with }Q(F):=\tfrac32\lambda |F|^2+\min \ol V-C_0,
\end{equation*}
satisfies 
\begin{equation}
\label{est:comparison2}
  \begin{aligned}
    \widehat V(F)\leq&\ol W(F)\qquad&&\text{for all }F,\\
    \widehat V(F)=&Q(F)\qquad&&\text{for all }|F|\geq R.
  \end{aligned}
\end{equation}
Indeed, for the inequality note that the lower bounds in \eqref{est:comparison} and \eqref{est:growth} yield
\begin{eqnarray*}
  Q(F)-\ol W(F)&\leq&   Q(F)-(\ol V(F)+\lambda\dist^2(F,U_\delta))\\
  &\leq& \lambda (\tfrac12|F|^2-\dist^2(F,U_\delta))-C_0,
\end{eqnarray*}
and thus the inequality in \eqref{est:comparison2} follows, since $Q(F)\leq \ol W(F)$ for all $F\in\R^{d\times d}$ provided  $C_0\lesssim 1$ is sufficiently large. The identity in \eqref{est:comparison2} follows, since $Q(F)-\ol V(F)\geq \frac{\lambda}{6}|F|^2-C_0-C'\geq0$ for $F$ sufficiently large.
We note that $\widehat V$ inherits the following properties from $\ol V$ and $W$:
\begin{enumerate}[(a)]
\item $\widehat V$ is strongly $2\lambda$-convex.
\item $\widehat V=\ol W$ on $U_\delta$; (indeed, this follows from the identities in \eqref{est:comparison2} and \eqref{est:comparison}).
\item $\widehat V$ is frame-indifferent and satisfies the growth condition
  \begin{equation*}
    \lambda|F|^2-C'\leq \widehat V(F)\leq 2\lambda|F|^2+C'
  \end{equation*}
  for a suitable constant $C'\lesssim 1$.
\item $\widehat V$ is $C^3$ in $U_{\text{good}}:=U_\delta\cup\{F\,:\,|F|\geq 2R\}$ with
  \begin{equation*}
    \|\widehat V\|_{C^3(U_{\text{good}})}\lesssim 1;
  \end{equation*}
  indeed, for $F\in U_\delta$, this follows from the regularity of $\ol W$ and (b), while for $|F|\geq 2R$ we deduce that $\widehat V$ is a quadratic function with $D^2\hat V[G,G]=D^2 Q[G,G]=3\lambda|G|^2$.
\end{enumerate}
Finally, we conclude by appealing to a regularization and gluing construction, see \cite{Ghomi} for a similar argument. We regularize $\widehat V$ via convolution with a compactly supported, frame-indifferent standard mollifier $\eta_\e=\e^{-d}\eta(\frac{\cdot}{\e})$ on $\R^{d\times d}$ $(\e>0)$, say
\begin{equation*}
  \eta(F):=\frac{1}{Z}
  \begin{cases}
    \exp(-\frac{1}{1-|F|^2})&\text{for }|F|\leq 1\\
    0&\text{else,}
  \end{cases}
\end{equation*}
where $Z>0$ is a normalizing constant.
 Since strong convexity is stable under convolution (with a positive function), the regularized potential $\widehat V_\e:=\widehat V*\eta_\e$ is smooth and strongly $2\lambda$-convex. In order to get the desired compatibility properties \eqref{WgeqV0} and \eqref{W=V0} (which are satisfied by $\widehat V$ but not necessarily by $\widehat V_\e$), we want to glue $\widehat V_\e$ and $\widehat V$ together, such that the resulting function $V$ is identical to $\widehat V$ in $U_{\delta/2}$ and differs from $\widehat V$ in $\{|F|\geq 3R\}$ only by a constant. To that end set $U_{\text{good}}':=U_{\delta/2}\cup\{F\,:\,|F|\geq 3R\}\Subset U_{\text{good}}$ and let $\rho$ denote a smooth cut-off function for $U_{\text{good}}'$ in $U_{\text{good}}$ (i.e. $0\leq \rho\leq 1$, $\rho=1$ on $U'_{\text{good}}$ and $\rho=0$ in $\R^{d\times d}\setminus U_{\text{good}}$). Since the sets $U_{\text{good}}$ and $U_{\text{good}}'$ are rotationally invariant, we may additionally assume that $\rho$ is frame-indifferent. Define
\begin{equation*}
  V(F):=\widehat V(F)+(1-\rho(F))(\widehat V_\e(F)-\widehat V(F)-C_\e),
\end{equation*}
where $C_\e:=\sup\{\widehat V_\e(F)-\widehat V(F)\,:\, F\in B_{3R}\,\}$. Then:
\begin{enumerate}[(a)]
\item[(e)] $V$ satisfies \eqref{WgeqV0} and \eqref{W=V0} (with $\delta$ replaced by $\delta/2$). Indeed, the latter follows from (b), since $V=\widehat V$ on $U_{\delta/2}$ thanks to the choice of the cut-off function $\rho$. The inequality \eqref{WgeqV0} follows from the combination of (b) and the identity $V=\widehat V$ on $\R^{d\times d}\setminus B_{3R}$ (which is true by the definition of the cut-off $\rho$), and the inequality $V\leq \widehat V$ on $B_{3R}$ (which is true thanks to the definition of $V$ and $C_\e$).

\item[(f)] For $\e=\e(d,\alpha,\mu,\delta)>0$ sufficiently small, the potential $V$ is strongly $\lambda$-convex. \\For the argument first note that the $2\lambda$-convexity of $Q$ and $\widehat V$, yields for $F$ with $\rho(F)\in\{0,1\}$:
  \begin{equation*}
    D^2V(F)[G,G]=
    \left.\begin{cases}
      D^2\widehat V(F)[G,G]&\text{if }\rho(F)=1\\
      D^2Q V(F)[G,G]&\text{if }\rho(F)=0\\
    \end{cases}\right\}\geq \lambda|G|^2.
  \end{equation*}
  In the other case, i.e. $\rho(F)\not\in\{0,1\}$, the claim follows, since $V$ (as a $C^2$ function) becomes arbitrarily close to either $\widehat V$ or $Q$, if $\e\ll 1$. For the precise argument consider
  \begin{equation*}
    H_\e(F):=V(F)-\widehat V(F)=(1-\rho(F))(\widehat V_\e(F)-\widehat V(F)-C_\e),
  \end{equation*}
  and note that $\rho(F)\not\in\{0,1\}$ implies  $F\in U'':=U_{\text{good}}\setminus U'_{\text{good}}$. The set $U''$ consists of two connected components, namely, $U''\cap U_\delta$ and $U''\setminus B_{2R}$. Since $\eta$ has compact support, we have (for $\e>0$ sufficiently small)
  \begin{equation*}
    H_\e=
    \begin{cases}
      (1-\rho)((\ol V*\eta_\e)-\ol V-C_\e)&\text{on }U''\cap U_{\delta},\\
      (1-\rho)((Q*\eta_\e)-Q-C_\e)&\text{on }U''\setminus B_{2R}.
    \end{cases}
  \end{equation*}
  We conclude in the case $U''\cap U_\delta$ (the argument in the case $U''\setminus B_{2R}$ is similar and left to the reader). From $\ol V=\ol W$ in $U_\delta\Supset U''\cap U_\delta$ we deduce that $\ol V$ is $C^3$ and thus
  \begin{eqnarray*}
    \limsup\limits_{\e\downarrow 0}\|(1-\rho)((\ol V*\eta_\e)-\ol V-C_\e)\|_{C^2(U''\cap U_{\delta})}=0.
  \end{eqnarray*}
  Hence, for $\e>0$ sufficiently small we get
  \begin{equation*}
    D^2V(F)[G,G]=D^2H(F)[G,G]+D^2\ol V(F)[G,G]\geq \frac{1}{2}D^2\ol V(F)[G,G]\geq \frac{\lambda}{2}|G|^2,
  \end{equation*}
  where the last inequality holds thanks to the $\lambda$-convexity of $\ol V$. 
\item[(g)] We finally note that $V\in \mathcal V_\beta$ for some $\beta>0$ that we eventually may choose (sufficiently small) only depending on $d,\alpha,\mu$. (Note that the quadratic growth condition is inherited from $\ol V$ and $Q$). Moreover $V$ is frame-indifferent, since $\ol V$, $Q$, $\rho$ and $\eta$ are frame-indifferent.
\end{enumerate}

\end{proof}

\section{Extended Corrector}\label{sec:corrector}
In this section, we introduce and establish regularity properties of the extended corrector $(\phi,\sigma)$ associated with a convex potential $V$. We first discuss the  definition of the extended corrector $(\phi,\sigma)$, and energy based regularity properties of $(\phi,\sigma)$  and the map $F\mapsto (\phi(F),\sigma(F))$. Secondly, we establish $W^{2,q}$-regularity (with $q>d$) based on non-variational methods (the implicit function theorem) and additional structural and smoothness assumptions imposed on $V$. Finally, we consider the special case when $V$ is given by the matching convex lower bound associated with the non-convex potential $W$. Let us anticipate that in the next section, in the proof of Theorem~\ref{T:1cell}, we shall see that the corrector $\phi(F)$ associated with the matching convex lower bound $V$ coincides with the corrector in the single-cell representation of $W_\ho(F)$ (for $F$ close to $\SO d$).
\medskip

For the above purpose it suffices to consider periodic convex potentials of the following type:
\begin{assumption}\label{ass:V}
 Fix $\beta>0$. We suppose that $V:\R^d\times\R^{d\times d}\to\R$ is a Caratheodory function that is $Y$-periodic in its first variable and $V(y,\cdot)\in \mathcal V_\beta$ for almost every $y\in Y$.
\end{assumption}
To a potential $V$ of class $\mathcal V_\beta$ and a direction $F\in \R^{d\times d}$ we associate an extended corrector $(\phi(F),\sigma(F))$ and a flux $J(F)$ according to the following lemma:
\begin{lemma}\label{L:phi}
  Suppose Assumption~\ref{ass:V} is satisfied. For every $F\in \R^{d\times d}$ there exist $\phi(F)\in H_{\per,0}^1(Y,\R^d)$, $\sigma(F)\in H_{\per,0}^1(Y,\R^{d\times d\times d})$ and $J(F)\in L_{\per,0}^2(Y,\R^{d\times d})$ that are uniquely defined by the following properties: 
  \begin{itemize}
  \item (Corrector). $\phi(F)$ satisfies
    \begin{equation}\label{eq:convexcorrector}
      \int_YV(y,F+\nabla \phi(y,F))\,dy= \min_{\phi\in H_{\per}^1}\int_YV(y,F+\nabla \phi(y))\,dy.
    \end{equation}
  \item (Flux). $J(F)$ satisfies
    \begin{equation}\label{def:J}
      J(F):=D V(\cdot,F+\nabla \phi(F))-\int_YDV(y,F+\nabla \phi(y,F))\,dy.
    \end{equation}
  \item (Flux-corrector). $\sigma(F)$ satisfies
    \begin{align}
      -\Delta \sigma_{ijk}(F) =& \partial_k J(F)_{ij}-\partial_j J(F)_{ik} \label{eq:sigma}\\
      \sigma_{ijk}(F)=&-\sigma_{ikj}(F)\label{prop:sigma1}\\
      -\partial_k \sigma_{ijk}(F)=&J(F)_{ij},\label{prop:sigma2}
    \end{align}
    where we use Einsteins summation convention in \eqref{prop:sigma2}.
  \end{itemize}
\end{lemma}
The argument for Lemma~\ref{L:phi} is rather standard. In particular, in the case of linear equations it is classical, e.g. see \cite[Section 7.2]{JKO-94}. For the reader's convenience we present the short proof below. Next, we present several regularity properties of the map $F\mapsto (\phi(F),\sigma(F))$. In order to shorten the notation, we introduce
\begin{equation}\label{def:L}
  \hat{\mathbb L}_{F}(x)=D^2 V(x,F+\nabla \phi(x,F)).
\end{equation}
Based on variational regularity methods, we obtain $H^1$-differentiability of the map $F\mapsto (\phi(F),\sigma(F))$:

\begin{lemma}\label{L:Dphi}
 Suppose Assumption~\ref{ass:V} is satisfied. The map $$F\mapsto (\phi(F),\sigma(F),J(F))$$ is of class $C^1(\R^{d\times d}, H_\per^1(Y)\times H_\per^1(Y)\times L_\per^2(Y))$. Moreover, for every $F,G\in\R^{d\times d}$ we have:
\begin{itemize}
 \item (Corrector). $D\phi(F)[G]=\partial_G\phi(F)$, where $\partial_G\phi(F)$ denotes the unique solution $\psi\in H_{\per,0}^1(Y)$ of
\begin{equation}\label{eq:dphi1}
 \divv( \hat{\mathbb L}_F(G+\nabla \psi)=0.
\end{equation}
\item (Flux).
\begin{equation}\label{D:J}
 DJ(F)[G] = \hat {\mathbb L}_F(G+\nabla \partial_G\phi(F)) - \int_Y \hat {\mathbb L}_F(G+\nabla \partial_G\phi(F)). 
\end{equation}
\item (Flux-corrector). $ D\sigma_{ijk}(F)[G]=\partial_G\sigma_{ijk}(F)$, where $\partial_G\sigma_{ijk}(F)$ denotes the unique solution in $\psi\in H_{\per,0}^1(Y,\R)$ of
\begin{equation}\label{eq:dgsigma:1}
 -\Delta \psi = \partial_k  (DJ(F)[G])_{ij}-\partial_j (DJ(F)[G])_{ik}. 
\end{equation}
\end{itemize}
\end{lemma}
The next lemma establishes (local) $W^{2,q}$-regularity $(q>d)$ for the extended corrector, provided the potential $V$ is sufficiently smooth and satisfies some additional structural assumptions:
\begin{lemma}\label{L:convexcorrector}
  Let $V$ satisfy Assumption~\ref{ass:V} and suppose that there exists $F_0\in\R^{d\times d}$ and $\delta>0$ s.t.
  \begin{equation*}
    \phi(F_0)=0\qquad\text{and}\qquad DV\in C^2(\R^d\times B_{\delta}(F_0)).
  \end{equation*}
  Then for any $q>d$ there exists $\bar{\bar\rho}>0$ such that 
  \begin{equation}\label{reg:convexcorrector}
    \|(F+\nabla\phi(F))-F_0\|_{L^\infty(Y)}<\delta\qquad\mbox{for all $F\in B_{\bar{\bar\rho}}(F_0)$},
  \end{equation}
  and the map $F\mapsto (\phi(F),\sigma(F))$ belongs to $C^1(B_{\bar{\bar\rho}}(F_0),W_{\per,0}^{2,q}(Y))$.
\end{lemma}
In the proof of the lemma we appeal to the implicit function theorem. The previous lemma is tailor made for an application to the \textit{matching convex lower bound} $V$ constructed in Section~\ref{sec:V}:
\begin{corollary}\label{C:corrector}
Let $W$ satisfy Assumption~\eqref{ass:W:1}, let $V:\R^d\times \R^{d\times d}\to\R$ denote the matching convex lower bound associated with $W$ via Corollary~\ref{C:wv} with parameters $\mu,\delta>0$, and let $(\phi,\sigma)$ denote the extended corrector associated with $V$. Then for any $q>d$ there exists $\bar{\bar\rho}>0$ such that 
\begin{equation*}
 \|\dist(F+\nabla\phi(F),\SO d)\|_{L^\infty(Y)}<\delta\qquad\mbox{for all $F\in U_{\bar{\bar\rho}}$},
\end{equation*}
and the map $F\mapsto(\phi(F),\sigma(F))$ belongs to  $C^1(U_{\bar{\bar\rho}},W_{\per,0}^{2,q}(Y))$.
\end{corollary}

The rest of this section is devoted to the proofs of Lemma~\ref{L:phi} -- Lemma~\ref{L:convexcorrector}, and Corollary~\ref{C:corrector}.

\begin{proof}[Proof of Lemma~\ref{L:phi}]
 
\step 1 Corrector $\phi$. For every $F\in\R^{d\times d}$ the minimization problem \eqref{eq:convexcorrector} admits a unique minimizer $\phi(F)\in H_{\per,0}^1(Y)$ as can be seen by appealing to the direct method, \eqref{vgrowth} and the strong convexity of $V(y,\cdot)$ for almost every $y\in Y$. Since $V$ is sufficiently smooth, see \eqref{vgrowthd1}, $\phi(F)$ is characterised as the unique $H_{\per,0}^1(Y)$ solution of the corresponding Euler-Lagrange equation
\begin{equation}\label{corrector:EL}
 \int_Y D V(y,F+\nabla \phi(F))[\nabla \eta]\,dy=0\quad\mbox{for all $\eta\in H_\per^1(Y)$.}
\end{equation}

\step 2 Flux $J$. Since $V(y,\cdot)\in \mathcal V_\beta$ and $\phi(F)\in H_\per^1(Y)$ for every $F\in\R^{d\times d}$, we obtain 
$$\int_Y|D V(y,F+\nabla\phi(F))|^2\,dy\stackrel{\eqref{vgrowthd1}}{\leq} \tfrac1\beta\int_Y(1+|F+\nabla\phi(F)|)^2\,dy<\infty,$$
and thus $J(F)\in L_\per^2(Y,\R^{d\times d})$.

\step 3 Flux-corrector $\sigma$. The right-hand side of \eqref{eq:sigma} is of divergence form as it can be rewritten as $\divv Q_{ijk}(F)$ with $Q_{ijk}=J(F)_{ij}e_k-J(F)_{ik}e_j\in L^2(Y,\R^d)$.  Hence, there exists a unique $H_{\per,0}^1(Y)$ solution for \eqref{eq:sigma}. The anti-symmetry property \eqref{prop:sigma1} follows directly from the equation \eqref{eq:sigma}. The identity \eqref{prop:sigma2} is a consequence of the periodicity of $J(F)$, $\divv(J(F))=0$, cf.~\eqref{corrector:EL}, and $\int_Y J(F)=0$. For convenience of the reader, we sketch the argument here: For every $\eta\in C_\per^\infty(Y)$ holds
\begin{align*}
 \sum_{k=1}^d\int_Y \partial_k\sigma_{ijk}(F)\Delta\eta=&\sum_{j,s=1}^d\int_Y \partial_s\sigma_{ijk}\partial_s\partial_k\eta\stackrel{\eqref{eq:sigma}}{=}\sum_{k=1}^d\int_YJ(F)_{ik}\partial_j\partial_k\eta- J(F)_{ij}\partial_k\partial_k\eta\\
 =& \sum_{k=1}^d \int_Y\langle J(F),\nabla(\partial_j\eta e_i)\rangle-\int_Y J(F)_{ij}\Delta\eta=-\int_Y J(F)_{ij}\Delta\eta.
\end{align*}
Hence, by a variant of Weyl's Lemma $z:=(\sum_k\partial_k\sigma_{ijk}(F))+J(F)_{ij}$ is smooth and harmonic. Combining this with the periodicity of $z$ and $\int_Yz=0$, we obtain $z\equiv0$.
\end{proof}

\begin{proof}[Proof of Lemma~\ref{L:Dphi}]
\step{1} Differentiability of $\phi$.

The proof closely follows \cite[Proof of Theorem 5.4]{Mueller93}. We first show that $F\mapsto \phi(F)$ is Lipschitz, then we prove the existence of all partial derivatives and conclude by showing the continuity of the partial derivatives. 

\textit{Substep 1.1} The map $F\mapsto \phi(F)$ is Lipschitz. 

We claim that for every $F,G\in\R^{d\times d}$ it holds
\begin{equation}\label{phi:lip}
 \|F+\nabla\phi(F)-(G+\nabla\phi(G))\|_{L^2(Y)}\leq \tfrac1{\beta^2} |F-G|.
\end{equation}
Indeed, strong convexity of $V$ in form of \eqref{ineq:strong-conv}, the Euler-Lagrange equation \eqref{corrector:EL}, and the growth condition \eqref{vgrowthd1} imply that
\begin{align*}
 &\beta\|F+\nabla\phi(F)-(G+\nabla\phi(G))\|_{L^2(Y)}^2\\
 &\leq \int_Y\left(DV(y,F+\nabla\phi(F))-DV(y,G+\nabla\phi(G))\right)[F+\nabla\phi(F)-G-\nabla\phi(G)]\,dy\\
 &= \int_Y\left(DV(y,F+\nabla\phi(F))-DV(y,G+\nabla\phi(G))\right)[F-G]\,dy\\
 &\leq\tfrac1\beta |F-G|\int_Y|F-G+\nabla(\phi(F)-\phi(G))|\leq \tfrac1\beta |F-G| \|F+\nabla\phi(F)-(G+\nabla\phi(G))\|_{L^2(Y)},
\end{align*}
and thus \eqref{phi:lip}. Combined with Poincar\'e's inequality, and $\int_Y\phi(F)=\int_Y\phi(G)=0$, we get
\begin{align}\label{phi:lip2}
 \|\phi(F)-\phi(G)\|_{H^1(Y)}\leq C(\beta,d)|F-G|.
\end{align}

\textit{Substep 1.2} Existence of partial derivatives.

Fix $F,G\in\R^{d\times d}$. We show  
\begin{equation}\label{lim:DGtphi:1}
 \partial_G^t\phi(F):=\tfrac1t(\phi(F+tG)-\phi(F))\to \partial_G\phi(F)\quad\mbox{in $H^1(Y)$ as $t\to0$,}
\end{equation}
where $\partial_G\phi(F)$ is the unique weak solution in $H_{\per,0}^1(Y)$ of \eqref{eq:dphi1}. For the argument, consider for $\eta\in H_\per^1(Y)$, $t\in\R$ and $s\in[0,1]$ the integral
\begin{equation*}
  \int_Y D V(y,F+\nabla\phi(F)+st(G+\nabla\partial_G^t\phi(F))\big)[\nabla \eta]\,dy=:f(s).
\end{equation*}
For $s=0$ and $s=1$ the integral turns into the left-hand side of the Euler-Lagrange equation for $\phi(F)$ and $\phi(F+tG)$, respectively. Hence, we have $f(1)=f(0)=0$, and  a Taylor expansion yields
\begin{equation}\label{eq:DGphitaylor1:1}
  0=\int_0^1f'(s)\,ds=t\int_Y\int_0^1\left\langle\hat{\mathbb L}_{F}^{s,t}(y)(G+\nabla\partial_G^t\phi(F)),\nabla\eta\right\rangle\,ds\,dy,
\end{equation}
where
\begin{equation}\label{def:Ns}
 \hat{\mathbb L}_{F}^{s,t}(y):=D^2V(y,N_G^{s,t}(F))\;\mbox{ with }\; N_G^{s,t}(F):=F+\nabla\phi(F)+st(G+\nabla \partial_G^t\phi(F)).
\end{equation}
Set $\psi_t:=\partial_G^t\phi(F)-\partial_G\phi(F)=(\partial_G^t\phi(F)+G)-(\partial_G\phi(F)+G)$. The ellipticity of $D^2V$ (cf. \eqref{vgrowthd2}), the equation for $\partial_G\phi(F)$ (cf. \eqref{eq:dphi1}), and  \eqref{eq:DGphitaylor1:1} imply
\begin{align*}
  \beta\int_Y|\nabla\psi_t|^2\leq& \int_Y\int_0^1\langle\hat{\mathbb L}_F^{s,t}(y)\nabla\psi_t,\nabla\psi_t\rangle\,ds\,dy\\
  =&\int_Y\int_0^1\langle(\hat{\mathbb L}_F(y)-\hat{\mathbb L}_F^{s,t}(y))(G+\nabla\partial_G\phi(F)),\nabla\psi_t\rangle\,ds\,dy.
\end{align*}
Hence, to conclude \eqref{lim:DGtphi:1}, it suffices to show that
\begin{equation*}
  \lim_{t\to\infty}\int_0^1\int_Y|(\hat{\mathbb L}_F-\hat{\mathbb L}_F^{s,t})(G+\nabla\partial_G\phi(F))|^2\,dsdy=0.
\end{equation*}
For the argument note that thanks to the uniform bound on $D^2V$ (cf. \eqref{vgrowthd2}), the integrand is dominated (up to a constant) by the integrable function $|G+\nabla\partial_G\phi(F)|^2$. Hence, it suffices to argue that from any subsequence, we can extract a subsequence $t_j\to 0$ s.t. 
\begin{equation*}
  \lim\limits_{j\to\infty}\hat{\mathbb L}_F^{s,t_j}(y)\to   \hat{\mathbb L}_F(y)\qquad\text{for almost every }y\in Y\text{ and every }s\in[0,1].
\end{equation*}
This follows from the Lipschitz-continuity of $\phi$, cf.~\eqref{phi:lip2}, which implies 
\begin{equation*}
 \limsup_{t\to0}\|\partial_G^t\phi(F)\|_{H^1(Y)}\leq C(\beta,d)|G|<\infty,
\end{equation*}
and thus $t\nabla \partial_G^t\phi(F)\to 0$ in $L^2(Y)$ and $N_G^{s,t}(F)\to F+\nabla\phi(F)$ in $L^2(Y)$ for every $s\in(0,1)$. By continuity of $D^2V$, the latter implies the claimed convergence (along subsequences).
\color{black}
\medskip

\textit{Substep 1.3} Continuity of partial derivatives.

We claim that $\partial_G\phi(F_n)\to\partial_G\phi(F)$ in $H^1(Y)$, whenever $F_n\to F$ in $\R^{d\times d}$. For the argument consider $\psi_n:=\partial_G\phi(F_n)-\partial_G\phi(F)$. The ellipticity of $\hat{\mathbb L}_{F_n}$ and \eqref{eq:dphi1} imply
\begin{align}\label{eq:substep13}
 \beta\int_Y|\nabla \psi_n|^2\leq& \int_Y\langle\hat{\mathbb L}_{F_n}\nabla \psi_n,\nabla\psi_n\rangle = \int_Y\langle(\hat{\mathbb L}_F-\hat{\mathbb L}_{F_n})(G+\nabla\partial_G\phi(F)),\nabla \psi_n\rangle.
\end{align}
Since $F\mapsto\phi(F)\in H_{\per,0}^1(Y)$ is Lipschitz-continuous it follows $\phi(F_n)\to\phi(F)$ in $H^1(Y)$ and thus, up to subsequence, $\hat{\mathbb L}_{F_n}\to\hat{\mathbb L}_F$ almost everywhere in $Y$. As in the previous substep, we conclude that $(\hat{\mathbb L}_F-\hat{\mathbb L}_{F_n})(G+\nabla\partial_G\phi(F))\to0$ in $L^2(Y)$. Hence, \eqref{eq:substep13} and Poincar\'e's inequality yield $\psi_n\to 0$ in $H^1(Y)$ and thus the claim.

\step{2} Differentiability of $J$ and continuity of $F\mapsto \partial_G J(F)$.

We first claim that $\tfrac1t(J(F+tG)-J(F))\to \partial_GJ(F)$ in $L^2(Y)$ as $t\to 0$ for any $F,G\in\R^{d\times d}$. For the argument denote by $\partial_GJ(F)$ the right-hand side of \eqref{D:J}. A Taylor expansion yields for almost every $y\in Y$ that
\begin{align*}
 \tfrac1t(J(F+tG)(y)-J(F)(y))=\int_0^1\hat{\mathbb L}_F^{s,t}(y)(G+\nabla \partial_G^t\phi(F)(y))\,ds-\int_Y\int_0^1\hat{\mathbb L}_F^{s,t}(G+\nabla \partial_G^t\phi(F))\,ds,
\end{align*}
where $\hat{\mathbb L}_F^{s,t}$ is defined as in \eqref{def:Ns}. Now, the claim follows from the convergence $\hat{\mathbb L}_F^{s,t}\to \hat{\mathbb L}_F$ (for almost every $y\in Y$ and all $s\in[0,1]$ up to a subsequence), the equiboundedness of $\hat{\mathbb L}_F^{s,t}$, and $\partial_G^t\phi(F)\to \partial_G\phi(F)$ in $H^1(Y)$.

The continuity of $F\mapsto\partial_G J(F)$ with respect to $L^2(Y)$ for every $G\in\R^{d\times d}$ is a straightforward consequence of the $H^1(Y)$ differentiability of $F\mapsto\phi(F)$ (proven in the previous step) and the equiboundedness of the tensor $\hat {\mathbb L}_F$.

\step{3} Differentiability of $\sigma$.

This is a direct consequence of the differentiability of $J$ and the fact that the map $T:L_{\per,0}^2(Y,\R^d)\to H_{\per,0}^1(Y,\R)$ given by $Tf:=u$ with $-\Delta u=\divv f$, is linear and bounded. Moreover, \eqref{eq:dgsigma:1} follows from \eqref{eq:sigma} and \eqref{D:J}.
 
\end{proof}

\begin{proof}[Proof of Lemma~\ref{L:convexcorrector}]

\step{1} Properties of the map $F\mapsto \phi(F)$.

Recall that $\phi(F)$ is characterised as the unique $H_{\per,0}^1(Y)$ solution of the corresponding Euler-Lagrange equation \eqref{corrector:EL}, which we recall in form of the identity 
\begin{equation*}
  -\divv\big(DV(\cdot,F_0+G+\nabla\varphi)\big)=0\qquad\text{ in the sense of distribution, where }G=F-F_0.
\end{equation*}
We prove that $F\mapsto\phi(F)$ satisfies the claimed properties by appealing to the implicit function theorem. To that end we define for $\rho>0$ the balls
\begin{align*}
 B_1(\rho):=\{G\in\R^{d\times d}\,|\, |G|<\rho\,\},\quad B_2(\rho):=\{\,\varphi\in W_{\per,0}^{2,q}(Y)\,|\, \|\varphi\|_{W^{2,q}(Y)}<\rho\,\}.
\end{align*}
By Sobolev embedding and $q>d$, we find $\rho'>0$ such that 
\begin{equation}\label{corrector:close}
  \|G+\nabla\varphi\|_{L^\infty(Y)}<\delta\quad\mbox{for all $G\in B_1(\rho')$ and $\varphi\in B_2(\rho')$}.
\end{equation}
Hence, thanks to the regularity assumption on $DV$, we may define
\begin{equation*}
  T:B_1(\rho')\times B_2(\rho')\to L_{\per,0}^q(Y),\qquad T(G,\varphi):=-\divv\big(DV(\cdot,F_0+G+\nabla\varphi)\big).
\end{equation*}
We claim that:
\begin{enumerate}[(a)]
\item $T\in C^1\big(B_1(\rho')\times B_2(\rho'), L_{\per,0}^q(Y)\big)$.\\
  \textit{Argument:} By assumption we have $DV\in C^2(\R^d\times B_{\delta}(F_0))$. Hence, the differentiability of $T$ follows from the differentiability of the corresponding composition operators, see Lemma~\ref{L:composition}, and \eqref{corrector:close}. In particular, we obtain   
  $$D_{\varphi} T(G,\varphi)[\phi]=-\divv(D^2V(\cdot,F_0+G+\nabla \varphi)[\nabla \phi]).$$
\item $T(0,0)=0$.\\
  \textit{Argument:} By assumption the functional
  \begin{equation*}
    H^1_{\per,0}(Y)\ni\varphi\mapsto \int_Y V(y,F_0+\nabla\varphi)\,dy
  \end{equation*}
  is minimized for $\varphi=\phi(F_0)=0$. Since the minimizer is characterized by the associated Euler-Lagrange equations, we deduce that $\int_Y DV(y,F_0)[\nabla\phi]\,dy=0$ for all $\phi\in H^1_{\per}(Y)$. Hence, periodicity and smoothness of $DV(\cdot,F_0)$ imply
  \begin{equation*}
    T(0,0)(y)=-\divv DV(y,F_0)=0\qquad\text{for all }y\in\R^d.
  \end{equation*}
\item $D_\varphi T(0,0):W_{\per,0}^{2,q}(Y)\to L_{\per,0}^q(Y)$ is invertible.\\
  \textit{Argument:} From (a), we deduce that $D_\varphi T(0,0)[\phi]=-\divv(D^2V(\cdot,F_0)[\nabla \phi])$. Since $V$ is strongly convex, we deduce that $D^2V(\cdot,F_0)\in C^1_\per(Y)$ is uniformly elliptic. Hence, we obtain by appealing to maximal $L^q$-regularity that $D_\varphi T(0,0)$ is invertible.
\end{enumerate}
By appealing to the implicit function theorem we conclude that there exists a map $\Lambda\in C^1(B_1(\bar{\bar\rho}), W_{\per,0}^{2,q}(Y))$ such that $T(G,\Lambda(G))=0$ and $\Lambda(G)\in B_2(\bar{\bar\rho})$. Since $\phi(F_0+G)=\Lambda(G)$, the claimed properties of the map $F\mapsto\phi(F)$ follow.

\step{2} Properties of the map $F\mapsto \sigma(F)$.

By combining
\begin{itemize}
 \item $\phi\in C^1(B_{\bar{\bar\rho}}(F_0), W_{\per,0}^{2,q}(Y))$ with $q>d$ and $\|\dist(F+\nabla \phi(F),\SO d)\|_{L^\infty}<\delta$, cf.~\eqref{reg:convexcorrector},
 \item $DV\in C^2(\R^d\times B_{\delta}(F_0))$,
\end{itemize}
we obtain that $F\mapsto J(F)$ is in $C^1(B_{\bar{\bar\rho}}(F_0),W_{\per,0}^{1,q}(Y))$, see~Lemma~\ref{L:composition}. Now, the claimed regularity property of $F\mapsto \sigma(F)$ follows from \eqref{eq:sigma} and the fact that the map $T:W_{\per,0}^{1,q}(Y)\to W_{\per,0}^{2,q}(Y)$ given by $f\mapsto T(f)$ with $\Delta T(f)=\divv f$ is bounded by maximal regularity and linear.
\end{proof}

\begin{proof}[Proof of Corollary~\ref{C:corrector}]
  Denote by $(\phi,\sigma,J)$ the extended corrector and the flux associated with $V$ according to Lemma~\ref{L:phi}. We argue that we can apply Lemma~\ref{L:convexcorrector} with $F_0=\Id$. Indeed, by the matching property, we have $V(y,\cdot)=W(y,\cdot)+\mu\det(\cdot)$ on $U_\delta$ (and thus on $B_\delta(\Id)$). Hence, the smoothness of $W$ close to $\R^d\times \SO d$, cf. \eqref{ass:wd2lip}, yields $DV\in C^2(\R^d\times B_\delta(\Id))$. Moreover, since $W(y,\cdot)$ is minimal at $\Id$, and $\det(\cdot)$ is a Null-Lagrangian, we conclude that 
\begin{equation*}
  H^1_{\per}(Y)\ni\varphi\mapsto \int_YW(y,\Id+\nabla\varphi)+\mu\det(\Id+\nabla\varphi)\,dy
\end{equation*}
is minimized by $\varphi=0$. Hence,
\begin{equation*}
  0=-\divv\Big(DW(\cdot,\Id)+\mu\det(\Id)\Big)=-\divv DV(\cdot,\Id)\qquad\text{on }\R^d,
\end{equation*}
and we conclude (by appealing to strong convexity) that the corrector at identity associated with $V$ vanishes, i.e. $\phi(\Id)=0$. Therefore, the assumptions of Lemma~\ref{L:convexcorrector} are satisfied, and we conclude that there exists $\bar{\bar\rho}>0$ s.t.  $F\mapsto (\phi(F),\sigma(F))$ is of class $C^1(B_{\bar{\bar\rho}}(\Id),W_{\per,0}^{2,q}(Y))$ and \eqref{reg:convexcorrector} holds (with $F_0=\Id$). By frame-indifference of $V$, i.e. \eqref{V:frame}, we have $\phi(F)=R\phi(R^TF)$ for all $R\in\SO d$ and $F\in\R^{d\times d}$ (and similarly for $\sigma$). Hence,  $F\mapsto (\phi(F),\sigma(F))$ is of class $C^1(U_{\bar{\bar\rho}},W_{\per,0}^{2,q}(Y))$, and we have \eqref{ref:convexcorrector2a}.

\end{proof}

\section{Single-cell formula, Proof of Theorem~\ref{T:1cell}}\label{sec:1cell}

The proof is organized as follows:
\begin{itemize}
\item In Step 1 we associated with $W$ a matching convex lower bound $V$ by appealing to Corollary~\ref{C:wv}, and we establish $C^2$-regularity of $V_{\ho}$. The latter only requires $H^1$-regularity of the corrector $\phi(F)$ (associated with $V$).
\item In Step 2 we show that $V_{\ho}$ is $C^3$-regular close to $\SO d$. This is done by appealing to \textit{Lipschitz regularity} of $\phi(F)$ for $F$ close to $\SO d$, see \eqref{ref:convexcorrector2a} and \eqref{ref:convexcorrector2b} below. These estimates are consequences of Corollary~\ref{C:corrector}.
\item In Step 3 we prove the validity of the single-cell formula. This is done by lifting the formula for the convex potential $V_{\ho}$ to the non-convex potential $W_\ho$. The argument relies on the matching property \eqref{W=V} and the Lipschitz regularity of $\phi(F)$, cf. \eqref{ref:convexcorrector2a}.
\item In Step 4 we prove  $C^3$-regularity of $W_{\ho}$ (close to $\SO d$) and establish the expansion of $W_{\ho}$. Again, this is done by lifting similar properties of $V_{\ho}$ to the level of $W_\ho$. In addition to \eqref{W=V} and \eqref{ref:convexcorrector2a}, the argument exploits the fact that $\det(\cdot)$ is a Null-Lagrangian.
\item In Step 5 we prove that $\phi(F)$ in the single-cell representation of $W_{\ho}(F)$ is unique, and in Step 6 we prove strong rank-one convexity of $W_\ho$ (close to $\SO d$).
\end{itemize}

\step 1 Matching convex lower bound $V$ and $C^2$-regularity  of $V_{\ho}$.

By appealing to Corollary~\ref{C:wv} we may associated with $W$ a matching convex lower bound $V$ with parameters $\delta,\mu,\beta>0$. We denote by $(\phi,\sigma)$ the extended corrector associated with $V$. 
We argue that $V_\ho\in C^2(\R^{d\times d})$ and claim that for all $F,G,H\in\R^{d\times d}$ we have
\begin{equation}\label{d2vhomgh}
  \begin{aligned}
    DV_\ho(F)[G]=&\,\int_Y DV(y,F+\nabla\phi(F))[G],\\
    D^2V_\ho(F)[G,H]=&\,\int_Y \left\langle\hat{\mathbb
        L}_F(G+\nabla\partial_G\phi(F)),H+\nabla\partial_H\phi(F)\right\rangle\,dy,
  \end{aligned}
\end{equation}
where $\hat{\mathbb L}_F$ and $\partial_G\phi$ are defined in \eqref{def:L} and Lemma~\ref{L:Dphi}. Indeed, the $C^2$-regularity is a consequence of \cite[Theorem 5.4]{Mueller93}, where periodic, strongly convex energy densities with quadratic growth are considered. Moreover, the first identity in \eqref{d2vhomgh} and
\begin{equation}\label{d2vhominf}
 D^2V_\ho(F)[G,G]=\inf\left\{\int_Y \langle \hat{\mathbb L}_F(y)(G+\nabla \psi),G+\nabla\psi\rangle\,dy\ |\ \psi\in H_\per^1(Y)\right\},
\end{equation}
follow from \cite[Theorem 5.4]{Mueller93} (and its proof) as well. Thus, by Lemma~\ref{L:Dphi} (cf. \eqref{eq:dphi1}) we get
\begin{equation*}
 D^2V_\ho(F)[G,G]=\int_Y \left\langle\hat{\mathbb L}_F(y)(G+\nabla\partial_G\phi(F)),G+\nabla\partial_G\phi(F)\right\rangle\,dy.
\end{equation*}
Combined with the polarization identity and \eqref{eq:dphi1} we obtain the second identity in \eqref{d2vhomgh}.

\step 2 $C^3$-regularity of $V_\ho$ close to $\SO d$.

To conclude $C^3$-regularity of $V_\ho$ we require Lipschitz regularity of the corrector $\phi$: By appealing to Corollary~\ref{C:corrector} we deduce that the corrector $\phi$ satisfies for some $\bar{\bar\rho}>0$ the Lipschitz estimate
\begin{equation}\label{ref:convexcorrector2a}
  \|\dist(F+\nabla\phi(F),\SO d)\|_{L^\infty(Y)}<\delta\qquad\text{for all }F\in U_{\bar{\bar\rho}},
\end{equation}
and (by Sobolev embedding and $q>d$) the regularity property 
\begin{equation}\label{ref:convexcorrector2b}
 F\mapsto \phi(F)\qquad\text{ is of class }C^1(U_{\bar{\bar\rho}},W^{1,\infty}(Y)).
\end{equation}
We claim that $V_\ho\in C^3(U_{\bar{\bar\rho}})$, and thus need to prove existence and continuity of all third order partial derivatives. 

\substep{2.1} Fix $F\in U_{\bar{\bar\rho}}$. For given $G,H,I\in\R^{d\times d}$, we claim that
\begin{eqnarray*}
  \partial_{G,H,I}^3V_\ho(F)&:=&\lim_{t\to0}\tfrac1t(D^2 V_\ho(F+tI)[G,H] - D^2 V_\ho(F)[G,H])\\
  &=&\int_YD^3V(y,F+\nabla\phi(F))[G+\nabla \partial_G(F),H+\nabla\partial_H\phi(F),I+\nabla\partial_I\phi(F)]\,dy\notag
\end{eqnarray*}
To shorten the presentation, for every $F,J\in\R^{d\times d}$, and $s,t\in\R$,  we set
\begin{eqnarray*}
  \Psi_J(F)&:=&J+\nabla \partial_J\phi(F),\\
  \partial_J^t\phi(F)&:=&\frac1t(\phi(F+tJ)-\phi(F)),\\
  N^{s,t}_J(F)&:=&F+\nabla\phi(F)+st(J+\nabla \partial_J^t\phi(F)).
\end{eqnarray*}
With this notation we may rephrase identity \eqref{d2vhomgh} as follows: For all $I\in\R^{d\times d}$ and $t\in\R$ we have
\begin{eqnarray*}
  D^2V_\ho(F)[G,H]&=&\int_Y D^2V(y,N_I^{s=0,t}(F))[\Psi_G(F),\Psi_H(F)]\,dy,\\
  D^2V_\ho(F+tI)[G,H]&=&\int_Y D^2V(y,N_I^{s=1,t}(F))[\Psi_G(F+tI),\Psi_H(F+tI)]\,dy.
\end{eqnarray*}
Hence, a Taylor expansion yields
\begin{align}\label{reg:whom:3a}
&\tfrac1t(D^2 V_\ho(F+tI)[G,H] - D^2 V_\ho(F)[G,H])\notag\\
&=\int_Y\int_0^1D^3V(y,N_I^{s,t}(F))[I+\nabla\partial_I^t\phi(F),\Psi_G(F+tI),\Psi_H(F+tI)]\,ds\,dy\notag\\
&\quad+\underbrace{\frac1t\int_Y\langle \hat{\mathbb L}_F(y)\Psi_G(F+tI),\Psi_H(F+tI)\rangle-\langle \hat{\mathbb L}_F(y)\Psi_G(F),\Psi_H(F)\rangle\,dy}_{=:r(t)}.
\end{align}
Thanks to $\phi\in C^1(U_{\bar{\bar\rho}},W_\per^{1,\infty}(Y))$, we have for every $J,K\in\R^{d\times d}$ and $s\in[0,1]$:
\begin{equation}\label{L:DGlimits}
\begin{split}
 &\lim_{t\to0}\|N_J^{s,t}(F)-(F+\nabla\phi(F))\|_{L^\infty(Y)}=0,\\
 &\lim_{t\to0}\|\partial_J^t\phi(F)-\partial_J\phi(F)\|_{W^{1,\infty}(Y)}=0,\\
 &\lim_{t\to0}\|\partial_J\phi(F+tK)-\partial_J\phi(F)\|_{W^{1,\infty}(Y)}=0.
\end{split}
\end{equation}
We conclude that the first integral on the right-hand side in \eqref{reg:whom:3a} converges to $\partial_{G,H,I}^3V_\ho(F)$ as $t\to0$ and it is left to show that $\lim_{t\to0}r(t)=0$. 
\smallskip

For the argument set $\psi_t:=\frac{1}{t}(\partial_G\phi(F+tI)-\partial_G\phi(F))$, so that $t\nabla\psi_t=\Psi_G(F+tI)-\Psi_G(F)$. Using \eqref{eq:dphi1} in the weak form of
\begin{equation*}
  \int_Y\langle \hat{\mathbb L}_F\Psi_G(F),\Psi_H(F+tI)-\Psi_H(F)\rangle=0,
\end{equation*}
we obtain
\begin{eqnarray*}
  r(t)&=&\frac{1}{t}\int_Y\langle \hat{\mathbb L}_F\Psi_G(F+tI),\Psi_H(F+tI)\rangle-
  \langle \hat{\mathbb L}_F\Psi_G(F),\Psi_H(F)\rangle\\
  &=&\frac{1}{t}\int_Y\langle \hat{\mathbb L}_F\big(\Psi_G(F+tI)-\Psi_G(F)\big),\Psi_H(F+tI)\rangle+\frac{1}{t}\int_Y\langle \hat{\mathbb L}_F\Psi_G(F),\Psi_H(F+tI)-\Psi_H(F)\rangle\\
  &=&\int_Y\langle \hat{\mathbb L}_F\nabla\psi_t,\Psi_H(F+tI)\rangle = \int_Y\langle \hat{\mathbb L}_F\Psi_H(F+tI),\nabla\psi_t\rangle,
\end{eqnarray*}
where the last identity holds by symmetry of $\hat{\mathbb L}_F$. 
To conclude Substep~2.1, it suffices to show that there exists $\psi\in H^1_{\per,0}(Y)$ s.t.
\begin{equation}\label{dphi:d2}
  \psi_t\wto \psi\quad\mbox{ weakly in $H^{1}(Y)$ as $t\to0$.}
\end{equation}
Indeed, by \eqref{L:DGlimits} we have $\Psi_H(F+tI)\to H+\nabla\partial_H\phi(F)$ strongly in $L^2(Y)$, and thus \eqref{dphi:d2} and \eqref{eq:dphi1} imply
\begin{eqnarray*}
  \lim\limits_{t\to 0}r(t)&=&\int_Y\langle \hat{\mathbb L}_F\Psi_H(F),\nabla\psi\rangle=0.
\end{eqnarray*}
In the remainder of this substep we prove \eqref{dphi:d2} where $\psi$ is given as the unique weak solution in $H_{\per,0}^1(Y)$ to
\begin{equation}\label{d2phi}
 -\divv \hat{\mathbb L}_F \nabla \psi=\divv D^3V(\cdot, F+\nabla \phi(F))[\Psi_I(F),\Psi_G(F)].
\end{equation}
Above (and below) we understand $D^3V(A)[B,C]$ as the $d\times d$-matrix with entries $\Big(D^3V(A)[B,C]\Big)_{ij}=D^3V(A)[B,C,e_i\otimes e_j]$. To see \eqref{dphi:d2} note that thanks to \eqref{eq:dphi1} we have
\begin{equation}\label{L:DGstep2}
 \begin{split}
   0=&\int_Y \langle \hat{\mathbb L}_{F+tI}\Psi_G(F+tI),\nabla\eta\rangle=\int\langle \hat{\mathbb L}_F\Psi_G(F),\nabla\eta\rangle,
 \end{split}
\end{equation}
for all $\eta\in H^1_\per(Y)$. Hence, a Taylor expansion yields
\begin{align}\label{eq:dg2taylor}
  0\stackrel{\eqref{L:DGstep2}}{=}&\frac1t\int_Y \langle \hat{\mathbb L}_{F+tI}(y)\Psi_G(F+tI),\nabla \eta\rangle\,dy\notag\\
  =&\frac1t\int_Y \langle \hat{\mathbb L}_F(y)\Psi_G(F+tI),\nabla \eta \rangle\,dy\notag\\
  &+\int_Y\int_0^1D^3V(y,N_I^{s,t}(F))[I+\nabla \partial_I^t\phi(F),\Psi_G(F+tI),\nabla\eta]\,ds\,dy\notag\\
  \stackrel{\eqref{L:DGstep2}}{=}&\int_Y \langle \hat{\mathbb L}_F(y)\nabla\psi_t,\nabla \eta \rangle\,dy\notag\\
  &\quad+\int_Y\int_0^1D^3V(y,N_I^{s,t}(F))[I+\nabla \partial_I^t\phi(F),\Psi_G(F+tI),\nabla\eta]\,ds\,dy,
\end{align}
Combined with \eqref{d2phi} we obtain
\begin{equation*}
  -\divv \hat{\mathbb L}_F\nabla(\psi_t-\psi)=\divv f_t
\end{equation*}
with right-hand side
\begin{align*}
 f_t:=&\int_0^1D^3V(\cdot,N_I^{s,t}(F))[I+\nabla \partial_I^t\phi(F),\Psi_G(F+tI)]\,ds\\
 &- D^3V(\cdot,F+\nabla \phi(F))[\Psi_I(F),\Psi_G(F)].
\end{align*}
Since $V$ is $C^3$, \eqref{L:DGlimits} implies that $f_t\to 0$ in $L_{\per}^\infty(Y)$, and thus \eqref{dphi:d2} (using that $\hat {\mathbb L}_F$ is uniformly elliptic).
\smallskip
\color{black}

\substep{2.2} 
It is left to show that the partial derivatives, given by $\partial_{G,H,I}^3V_\ho$ are continuous in $U_{\bar{\bar\rho}}$. Since $F\mapsto \phi(F)$ is in $C^1(U_{\bar{\bar\rho}},W_{\per,0}^{1,\infty})$ and $\partial_G\phi(F)=D\phi(F)[G]$ (cf.~Lemma~\ref{L:Dphi} and \ref{L:convexcorrector}), we have for every $F\in U_{\bar{\bar\rho}}$, $G\in \R^{d\times d}$ and $(F_m)\subset \R^{d\times d}$ with $\lim_{m\to\infty} F_m=0$ that 
\begin{equation*}
 \lim_{m\to\infty}(\|\phi(F_m)-\phi(F)\|_{W^{1,\infty}(Y)}+\|\partial_G\phi(F_m)-\partial_G\phi(F)\|_{W^{1,\infty}(Y)})=0.
\end{equation*}
Hence, from $V(y,\cdot)\in C^3(\R^{d\times d},\R)$ for almost every $y\in Y$ we deduce that $\partial_{G,H,I}^3V_\ho\in C(U_{\bar{\bar\rho}})$ and thus $V_\ho\in C^3(U_{\bar{\bar\rho}})$.

\step{3} Single-cell formula; proof of (a).

In this step we prove that for all $F\in U_{\bar{\bar\rho}}$ we have
\begin{equation}\label{pf:onecell}
  W_\ho(F)=W\super1_\ho(F)=\int_YW(y,F+\nabla\phi(F))\,dy=V_\ho(F)-\mu\det(F),
\end{equation}
where $\phi(F)$ denotes the corrector associated with $V$. We start our argument with the claim that
\begin{equation}\label{WhgeqVh}
 W_\ho(F)\geq V_\ho(F)-\mu\det F\quad\mbox{for all $F\in\R^{d\times d}$.}
\end{equation}
Indeed, for $k,m\in\N$, we find $\phi\in W_\per^{1,p}(kY)$ such that 
$$ W_\ho\super{k}(F)+\tfrac1m\geq\fint_{kY}W(y,F+\nabla \phi(y))\,dy.$$
From \eqref{WgeqV} in Corollary~\ref{C:wv}, we learn that the right-hand side is bounded below by
\begin{align*}
  \fint_{kY}V(y,F+\nabla \phi)\,dy-\mu\fint_{kY}\det (F+ \nabla \phi(y))\,dy.
\end{align*}
Since $\det(\cdot)$ is a Null-Lagrangian and $\phi\in W_\per^{1,p}(kY)$ with $p\geq d$, we deduce that $\fint_{kY}\det (F+ \nabla \phi(y))\,dy=\det(F)$ (see e.g. \cite[Theorem 8.35]{Dac07}). Hence, we get 
\begin{align*}
 W_\ho\super{k}(F)+\tfrac1m\geq\inf_{\psi\in H_\per^1(kY)}\fint_{kY}V(y,F+\nabla\psi)\,dy-\mu\det(F).
\end{align*}
By convexity and quadratic growth of $V$, we deduce that the infimum over $H_\per^1(kY)$ can be replaced by an infimum over $H_\per^1(Y)$, see \cite[Lemma~4.1]{Mueller87}. Hence, we arrive at
$$W_\ho\super{k}(F)+\tfrac1m\geq V_\ho(F)-\mu\det (F).$$
By first taking the limit $m\to\infty$ and then the infimum over $k\in\N$, we obtain \eqref{WhgeqVh}.
\smallskip

Next, fix $F\in U_{\bar{\bar\rho}}$ and recall that the corrector $\phi(F)$ associated with the convex potential $V$ satisfies
$$V_\ho(F)=\int_Y V(y,F+\nabla \phi(F))\,dy\quad\mbox{and}\quad \|\dist(F+\nabla \phi,\SO d)\|_{L^\infty(Y)}<\delta.$$
Hence, we get
\begin{align*}
 W_\ho(F)&\leq W_\ho^{(1)}(F)\leq \int_YW(y,F+\nabla \phi(F))\,dy\\
 &\stackrel{\eqref{W=V}}{=}\int_YV(y,F+\nabla \phi(F))-\mu\det(F+\nabla \phi(F))\,dy=V_\ho(F)-\mu\det (F)\stackrel{\eqref{WhgeqVh}}{\leq}W_\ho(F),\notag
\end{align*}
and thus \eqref{pf:onecell}.

\step 4 Regularity and quadratic expansion of $W_{\ho}$; proof of (c).

From $V_\ho\in C^3(U_{\bar{\bar\rho}})$, cf. Step 2, identity \eqref{pf:onecell}, and the smoothness of $F\mapsto \det(F)$, we deduce that $W_\ho\in C^3(U_{\bar{\bar\rho}})$. We also obtain for all $F\in U_{\bar{\bar\rho}}$ and $G\in\R^{d\times d}$ the identities
\begin{eqnarray*}
  DW_{\ho}(F)[G]&=&DV_\ho(F)[G]-\mu D\det(F)[G],\\
  D^2W_{\ho}(F)[G,G]&=&D^2V_\ho(F)[G,G]-\mu D^2\det(F)[G,G].
\end{eqnarray*}
Combined with the formula for $DV_\ho$, the Lipschitz bound \eqref{ref:convexcorrector2a}, and the matching property \eqref{W=V}, we obtain
\begin{align}\label{pf:expansion1}
  &DW_{\ho}(F)[G] = \int_YDV(y,F+\nabla\phi(F))[G]-\mu D\det(F)[G]\\\notag
                &=\int_YDW(y,F+\nabla\phi(F))[G]+\mu \underbrace{\big(\int_YD\det(F+\nabla\phi(F))[G]-D\det(F)[G]\big)}_{=:r_1}.
\end{align}
In analogy to \eqref{def:L} we set $\mathbb L_F(y):=D^2W(y,F+\nabla \phi(y,F))$, and note that \eqref{ref:convexcorrector2a} and \eqref{W=V} yield
\begin{equation}\label{pf:LhatL}
  \mathbb L_F(y)=\hat{\mathbb L}_F(y)-\mu D^2\det(F+\nabla \phi(y,F)).
\end{equation}
From the formula for $D^2V_\ho$ (cf. \eqref{d2vhomgh}), we deduce that
\begin{align}
  \label{pf:expansion2}
  &D^2W_{\ho}(F)[G,G] = D^2V_{\ho}(F)[G,G]-\mu D^2\det(F)[G,G]\\\notag
  &=\int_Y \langle \hat{\mathbb L}_F\big(G+\nabla\partial_G\phi(F)\big),G+\nabla\partial_G\phi(F)\rangle-\mu D^2\det(F)[G,G]\\\notag
  &=\int_Y \langle \mathbb L_F\big(G+\nabla\partial_G\phi(F)\big),G+\nabla\partial_G\phi(F)\rangle\\\notag
  &\quad+\mu\underbrace{\big(\int_Y D^2\det(F+\nabla\phi(F))[G+\nabla\partial_G\phi(F),G+\nabla\partial_G\phi(F)]-D^2\det(F)[G,G]\big)}_{=:r_2}.
\end{align}
Next, we argue that $r_1=r_2=0$ follows the fact that the determinant is a Null-Lagrangian. We only present the argument for $r_2=0$ (the argument for $r_1=0$ is simpler).
It suffices to show that for every $F,G\in\R^{d\times d}$, and $\psi\in H_\per^1(Y)$ it holds
\begin{equation}\label{C:d2whom1}
 \int_Y D^2\det(F+\nabla\phi(F))[G+\nabla \psi,G+\nabla\psi]=D^2\det(F)[G,G].
\end{equation}
Since, the left-hand side in \eqref{C:d2whom1} is continuous in $\psi$ (with respect to strong convergence in $H^1(Y)$), it suffices to consider $\psi\in C_\per^\infty(Y)$. In that case we have (exploiting the fact that $\det(\cdot)$ is a Null-Lagrangian): 
\begin{align*}
 &\frac1t\frac1s(\det(F+(s+t)G)-\det(F+sG)-tD\det(F)[G])\\
 &=\frac1t\frac1s\int_Y\big(\det(F+\nabla\phi(F)+(s+t)(G+\nabla\psi))-\det(F+\nabla\phi(F)+s(G+\nabla\psi))\\
 &\qquad\qquad-tD\det(F+\nabla\phi(F))[G+\nabla\psi]\big).
\end{align*}
By sending successively $t\to0$ and $s\to0$, we obtain \eqref{C:d2whom1}.

From $r_1=0$ and \eqref{pf:expansion1}, we directly obtain the asserted identity for $DW_{\ho}(F)$. The argument for the asserted identity for $D^2W_{\ho}(F)$ is as follows:
\begin{align*}
  &\inf_{\psi\in H^1_{\per}(Y)}\int_Y\langle \hat{\mathbb L}_F(G+\nabla\psi),G+\nabla\psi\rangle\\
  &\stackrel{\eqref{pf:LhatL}}{=}\, 
    \inf_{\psi\in H^1_{\per}(Y)}\left\{\int_Y\langle {\mathbb L}_F(G+\nabla\psi),G+\nabla\psi\rangle-\mu D^2\det(F+\nabla\phi(F))[G+\nabla\psi,G+\nabla\psi]\right\}\\
  &\stackrel{\eqref{C:d2whom1}}{=}\,                                          
    \inf_{\psi\in H^1_{\per}(Y)}\left\{\int_Y\langle {\mathbb L}_F(G+\nabla\psi),G+\nabla\psi\rangle\right\}-\mu D^2\det(F)[G,G]\\
  &\stackrel{\eqref{d2vhominf}}{=}\,D^2V_\ho(F)[G,G]-\mu D^2\det(F)[G,G]\stackrel{\eqref{pf:expansion2},r_2=0}{=}D^2W_\ho(F)[G,G].
\end{align*}

\step 5 Uniqueness and regularity of the corrector; proof of (b).
\smallskip

Let $F\in U_{\bar{\bar\rho}}$. From \eqref{pf:onecell} we learn that
\begin{equation}\label{Whcor}
  W_\ho(F)=\int_YW(y,F+\nabla\tilde\phi)\,dy
\end{equation}
is satisfied for $\tilde\phi=\phi(F)$. If $\tilde\phi\in W^{1,p}_{\per,0}(Y)$ denotes another function satisfying \eqref{Whcor}. Then 
\begin{eqnarray*}
  W_\ho(F)&\stackrel{\eqref{Whcor}}{=}&\int_YW(y,F+\nabla\tilde\phi)\,dy\stackrel{\eqref{WgeqV}}{\geq}\int_Y V(y,F+\nabla\tilde\phi)-\mu\det(F)\\
  &\geq& V_\ho(F)-\mu\det F\stackrel{\eqref{pf:onecell}}{=}W_\ho(F).
\end{eqnarray*}
Hence, $\tilde\phi$ is a minimizer of $\psi\mapsto \int_YV(y,F+\nabla\psi)$. By strong convexity of $V$, minimizers are unique and we deduce that $\tilde\phi=\phi(F)$. This proves uniqueness of $\phi(F)$. For the asserted regularity of $\phi(F)$ see \eqref{ref:convexcorrector2a}.

\smallskip

\step{6} Strong rank-one convexity of $W_{\ho}$; proof of (d).

Statement (d) is a consequence of the strong $\beta$-convexity of $V$ and the fact that the determinant is a Null-Lagrangian and thus rank-one affine. Indeed, \eqref{d2vhominf} and $V(y,\cdot)\in\mathcal V_\beta$ for almost every $y\in Y$ yield
$$\beta|G|^2\leq D^2V_\ho(F)[G,G]\quad\mbox{for all $F,G\in\R^{d\times d}$}.$$
Hence, the claim \eqref{limit:nondegenerate} follows by $D^2W_\ho(F)[G,G]=D^2V_\ho(F)[G,G]-\mu D^2\det(F)[G,G]$ for all $F\in U_{\bar{\bar\rho}}$ and $G\in\R^{d\times d}$ and $D^2\det(F)[a\otimes b,a\otimes b]=0$ for all $F\in\R^{d\times d}$ and $a,b\in\R^d$. 
\qed

\section{Homogenization error}\label{sec:hom}

Throughout this section we suppose that
\begin{itemize}
\item $W$ satisfies Assumptions~\ref{ass:W:1} for some $\alpha>0$ and $p\geq d$.
\end{itemize}
By appealing to the results of the previous sections, we can find constants $\beta,\delta,\mu,\bar{\bar\rho}>0$, $q>d$  (which are from now on fixed) such that the following properties hold:
\begin{itemize}
\item There exists a matching convex lower bound $V$ satisfying the properties of Corollary~\ref{C:wv} with parameters $\beta,\delta,\mu>0$. (Existence of $V$ follows by Corollary~\ref{C:wv}).
\item The map $F\mapsto (\phi(F),\sigma(F))$, where $(\phi,\sigma)$ denotes the extended corrector associated with $V$, is of class  $C^1(U_{\bar{\bar\rho}},W_{\per,0}^{2,q}(Y))$. (This can be achieved by appealing to Lemma~\ref{L:convexcorrector}).
\item The conclusion of Theorem~\ref{T:1cell} holds for the parameter $\bar{\bar\rho}$. In particular, we have $W_{\ho}(F)=\int_YW(y,F+\nabla\phi(F))$ for $F\in U_{\bar{\bar\rho}}$.
\item The convex homogenized integrand satisfies $V_\ho\in C^2(\R^{d\times d})$ and $\beta |G|^2\leq D^2V_\ho(F)[G,G]\leq \frac1\beta |G|^2$, for every $F,G\in \R^{d\times d}$. (See Step~2 in the proof of Theorem~\ref{T:1cell}, and \eqref{d2vhominf}).
\end{itemize}
Finally, we suppose that 
\begin{itemize}
\item $A$ is a bounded domain of class $C^2$.
\end{itemize}
We consider the energy functionals
\begin{eqnarray*}
  \mathcal E_{\e}:H^1(A)\to(-\infty,\infty),\qquad\mathcal E_\e(u)&:=&\int_A V\left(\frac{x}{\e},\nabla u(x)\right)-f(x)\cdot u(x)\,dx,\\
  \mathcal I_{\e}:H^1(A)\to(-\infty,\infty),\qquad\mathcal I_\e(u)&:=&\int_A W\left(\frac{x}{\e},\nabla u(x)\right)-f(x)\cdot u(x)\,dx,
\end{eqnarray*}
and the associated homogenized functionals
\begin{eqnarray}
  \mathcal E_{\hom}:H^1(A)\to(-\infty,\infty),\qquad \mathcal E_{\hom}(u)&:=&\int_A V_{\hom}(\nabla u(x))-f(x)\cdot u(x)\,dx,\label{ene:convexhom}\\
  \mathcal I_{\hom}: W^{1,p}(A)\to (-\infty,\infty],\qquad\mathcal I_{\hom}(u)&:=&\int_A W_{\hom}(\nabla u(x))-f(x)\cdot u(x)\,dx.\label{ene:nonconvexhom}
\end{eqnarray}
Our main result in this section is the upcoming quantitative two-scale expansion. Theorem~\ref{T:NC:intro} is a special case of it. To ease the presentation, throughout this section we write $\lesssim$ whenever it holds $\leq$ up to a multiplicative constant only depending on $\alpha,\beta,\delta,\mu$ and $A$, and we write $a \leq C(b)c$ whenever $a\leq b$ holds up to a positive multiplicative constant $C(b)$  which depends on $b$ only.

\begin{theorem}[Homogenization error - non-convex case]\label{T:NC}
Fix $r>d$. Then there exists $\bar \rho\ll 1$  such that following statement holds:
\smallskip

If the data $f\in L^r(A)$ and  $g\in W^{2,r}(A)$ is small in the sense of
\begin{equation*}
 \Lambda(f,g,g_0):=\|f\|_{L^r(A)}+\|g-g_0\|_{W^{2,r}(A)}+\|\dist(\nabla g_0,\SO d)\|_{L^\infty(A)}<\bar\rho,
\end{equation*}
where $g_0\in W^{2,r}(\R^d)$ is assumed to satisfy
\begin{equation}\label{NC:lin_bc}
 -\divv(DW_{\hom}(\nabla g_0))=0\qquad\text{in }\mathcal D'(A),
\end{equation}
(e.g. $g_0(x)=x$),  then:
\begin{enumerate}[(a)]
 \item The functional \eqref{ene:nonconvexhom} admits a unique minimizer $u_0\in g+ W_0^{1,p}(A)$. It satisfies the nonlinear estimates
 \begin{subequations}
  \begin{eqnarray}
   \label{NC:NL:1}
   \|\dist(\nabla u_0,\SO d)\|_{L^\infty(A)}&<& \bar{\bar \rho},\\
   \label{NC:NL:2}
   \|\dist(\nabla u_0,\SO d)\|_{L^\infty(A)}+\|u_0-g_0\|_{W^{2,r}(A)}&\lesssim&C(\bar \rho)\Lambda(f,g,g_0).
  \end{eqnarray}
 \end{subequations}
 \item $u_0$ is characterized as the unique minimizer in $g+H_0^1(A)$ of the auxiliary functional \eqref{ene:convexhom}.
 \item For any $w_\e\in g+W_0^{1,p}(A)$ and every $\e\in(0,1)$ it holds
    \begin{eqnarray*}
      &&\|w_\e-u_0\|_{L^2(A)}+\|w_\e-(u_0+\e\nabla\phi(\tfrac{\cdot}\e,\nabla u_0))\|_{H^1(A)} \\
       &&\lesssim C(\bar\rho)\left(\e^\frac12\Lambda(f,g,g_0) + \e(1+\|\nabla^2g_0\|_{L^r(A)}^\frac{r}{r-d})(\|\nabla^2g_0\|_{L^r(A)}+\Lambda(f,g,g_0) \right)\\
       &&\qquad +\left(\mathcal I_\e(w_\e)-\inf_{g+W_0^{1,p}} \mathcal I_\e\right)^\frac12.
    \end{eqnarray*}
  \end{enumerate}
\end{theorem}

\begin{remark}\label{Rem:bending}

The simplest choice for $g_0$ are rigid-motions, i.e. $g_0(x)=Rx+c$ where $R\in\SO d$ and $c\in\R^d$. Indeed, since $W_{\hom}$ is minimized at $\SO d$ the equation \eqref{NC:lin_bc} holds trivially. In this case Theorem~\ref{T:NC} implies that minimizers of $\mathcal I_\ho$ and $\mathcal I_\e$ are close to a rigid motion.  However, the conclusion of Theorem~\ref{T:NC} is valid also in more general situations which ensure that, roughly speaking, minimizers of $\mathcal I_\ho$ are locally close to isometries. In the following, we give a non-trivial example of such a situation, which corresponds to bending of a thin, planar beam: Let $y=(                                                                                                                                                                                                                                                                                                                                                                                                                                                                                                                                                                                                                 y_1,y_2)^T                                                                                                                                                                                                                                                                                                                                                                                                                                                                                                                                                                                                                                                         \in C^\infty(\R,\R^2)$ satisfy $|y'(x_1)|=1$ for all $x_1\in\R$, and consider $\tilde g_0\in C^\infty(\R^2,\R^2)$ defined via $\tilde g_0(x_1,x_2)=y(x_1)+x_2(                                                                                                                                                                                                                                                                                                                                                                                                                                                                                                                                                                                                                                                                                                                                                                                                                                                                                                                                                                                                                                                                                                                                                                                                                                                                                                                                                                                                                                                                                                                                                                                                                                                                                                                                                                                                                                                                                                                                                                                                                                                                                                                                           -y_2'(x_1),y_1'(x_1))^T$. Note that $\nabla \tilde g_0(x_1,x_2)=R(x_1)+x_2M(x_1)$, where $R\in C^\infty(\R,\SO 2)$ and $M\in C^\infty(\R,\R^{2\times 2})$, where $M(x_1)$ has rank-one for all $x_1\in\R$. Set $A_h:=(-1,1)\times (-3h,3h)$ and consider the variational problem 
$$\min\left\{\int_{A_h}V_\ho (\nabla u)\,|\,u\in \tilde g_0+H_0^1(A_h)\right\},$$
where $g_0$ denotes the corresponding unique minimizer. We show that for $h>0$ sufficiently small and $A_h':=(-\frac12,\frac12)\times (-\frac{h}{2},\frac{h}{2})\Subset A_h$, the minimizer $g_0\in \tilde g_0+H_0^1(A_h)$ satisfies
\begin{itemize}
 \item[(a)] $g_0\in W^{2,q}(A_h')$ for some $q>2$,
 \item[(b)] $\divv DW_\ho(\nabla g_0)=0$ on $A_h'$.
\end{itemize}
The higher (interior) differentiability of $g_0$ is standard and can be obtained by the difference quotient method, which yields local $W^{2,2}$ regularity, and an application of a reverse H\"older inequality (see e.g~\cite{GM12}). In particular, we have the following: There exists $q>2$ such that $g_0\in W_\loc^{2,q}(A_h)$ and for every $x\in A_h$ with $Q_{2h}(x)\Subset A_h$, where $Q_h(x):=x+(-h,h)^2$, it holds
\begin{equation}\label{bending:gehring}
 \left(\fint_{Q_{\frac{h}2}(x)}|\nabla^2 g_0|^q\right)^\frac1q\lesssim \left( \fint_{Q_{h}(x)}|\nabla^2 g_0|^2\right)^\frac12\lesssim \frac1h \left(\fint_{Q_{2h}(x)}|\nabla g_0-(\nabla g_0)_{Q_{2h}(x)}|^2 \right)^\frac12,
\end{equation}
where $(\nabla g_0)_{Q}:=\fint_Q \nabla g_0$ and here and for the rest of this remark $\lesssim$ means $\leq$ up to a multiplicative constant which is independent of $h$ and $x$.

Let us now discuss (b). Since $W_\ho=V_\ho-\mu \det$ on $U_{\bar{\bar \rho}}$, $\divv DV_\ho (\nabla g_0)=0$ by the minimiality of $g_0$, and $\divv D\det(\nabla g_0)=0$ since $\det$ is a Null-Lagrangian, it suffices to show $\|\dist(\nabla g_0,\SO 2)\|_{L^\infty(A_h')}<\bar{\bar \rho}$. First, we observe that $\nabla g_0$ is close to the set of rotations in the sense of
\begin{equation}\label{bending:l2}
 \int_{A_h} \dist^2(\nabla g_0,\SO 2)\lesssim h^3.
\end{equation}
Indeed, a combination  of $V_\ho(\nabla \tilde g_0)=W_\ho(\nabla \tilde g_0)-\mu\det(\nabla \tilde g_0)$ on $A_h$ (for $h>0$ sufficiently small), the fact that $DW_{\ho}(R)=0$ for all $R\in \SO d$ and thus  $\|DW_{\ho}(\nabla \tilde g_0)\|_{L^\infty(A_h)}\lesssim \|\dist(\nabla \tilde g_0,\SO d)\|_{L^\infty(A_h)}\lesssim h$, and
\begin{align*}
 0\geq& \int_{A_h}V_\ho(\nabla g_0)-V_\ho(\nabla \tilde g_0)\stackrel{\eqref{ineq:strong-conv2}}{\geq} \int_{A_h}DV_\ho(\nabla \tilde g_0)[\nabla g_0-\nabla \tilde g_0]+\frac\beta2 |\nabla g_0-\nabla \tilde g_0|^2\\
 \geq& \int_{A_h} DW_\ho(\nabla \tilde g_0)[\nabla g_0-\nabla \tilde g_0]+\frac\beta2 |\nabla g_0-\nabla \tilde g_0|^2-\mu\underbrace{\int_A D\det(\nabla \tilde g_0)\nabla g_0-\nabla \tilde g_0]}_{=0},
\end{align*}
implies via Young's inequality
$$ \|\nabla g_0-\nabla \tilde g_0\|_{L^2(A_h)}^2\lesssim h^3.$$
Inequality \eqref{bending:l2} follows since $\|\dist(\nabla \tilde g_0,\SO 2)\|_{L^2(A_h)}^2\lesssim h^3$ by construction. 

Next, we combine \eqref{bending:gehring} with the geometric rigidity estimate (see \cite[Theorem 3.1]{FJM02}) to obtain for all $x\in A_h$ with $Q_{2h}(x)\Subset A_h$ that
\begin{align}\label{bending:3}
 \left(\fint_{Q_{\frac{h}2}(x)}|\nabla^2 g_0|^q\right)^{\frac{1}{q}}\lesssim& \frac1{h} \left(\fint_{Q_{2h}(x)}|\nabla g_0-(\nabla g_0)_{Q_{2h}(x)}|^2\right)^\frac12 \leq \frac1{h} \left(\inf_{R\in\SO 2}\fint_{Q_{2h}(x)}|\nabla g_0-R|^2\right)^\frac12\notag\\
  \lesssim&  \frac1{h^2}\left(\int_{Q_{2h}(x)}\dist^2 (\nabla g_0,\SO 2)\right)^\frac12\leq \frac1{h^2}\left(\int_{A_{h}}\dist^2 (\nabla g_0,\SO 2)\right)^\frac12.
\end{align}
Hence, Morrey's inequality combined with \eqref{bending:l2} and \eqref{bending:3} yields
\begin{equation*}
 \|\nabla g_0-(\nabla g_0)_{Q_{\frac{h}2}(x)}\|_{L^\infty(Q_{\frac{h}2}(x))}\lesssim h^{1-\frac{2}{q}}\|\nabla^2 g_0\|_{L^q(Q_{\frac{h}2}(x))}\lesssim h^\frac12.
\end{equation*}
By geometric rigidity, we find for any $Q_{\frac{h}2(x)}\subset A_h$ a rotation $R_h(x)\in \SO 2$ such that
\begin{equation*}
 |\fint_{Q_{\frac{h}2(x)}}\nabla g_0-R_h(x)|^2\leq \frac{1}{h^2}\int_{Q_{\frac{h}2}(x)}|\nabla g_0-R_h(x)|^2\lesssim\frac1{h^2}\int_{A_h}\dist^2(\nabla g_0,\SO 2)\stackrel{\eqref{bending:l2}}{\lesssim} h,
\end{equation*}
and thus $\|\dist(\nabla g_0,\SO 2)\|_{L^\infty(A_h')}<\bar{\bar\rho}$ for $h>0$ sufficiently small.
\end{remark}

\bigskip

The proof of Theorem~\ref{T:NC} is presented in Section~\ref{S:NC}. At several place we lift results obtained for the convex functionals $\mathcal E_\e$ and $\mathcal E_\ho$ to the level of the non-convex functionals. Therefore, we first discuss the convex case.

\subsection{Convex case}

In this section, we provide a quantitative two-scale expansion for the convex problem $\mathcal E_\e$ for small loads:
\begin{proposition}[Homogenization error - convex case]\label{P:C}
  For every $\bar\rho< \bar{\bar\rho}$ the following statement holds:
  \begin{itemize}
  \item Let $f\in L^2(A)$, and denote by $u_0\in H^2(A)$ the unique minimizer of $\mathcal E_{\hom}$, given in \eqref{ene:convexhom}, subject to its own boundary conditions.
  \item Let $u_\e\in H^1(A)$ be a minimizer of $\mathcal E_\e$ subject to $u_\e-u_0\in H_0^1(A)$.
  \item Let $v_\e$ denote the two-scale expansion,
    \begin{equation*}
      v_\e(x)=u_0(x)+\e\eta(x)\phi(\tfrac{x}{\e},\nabla u_0(x))
    \end{equation*}
    where  $\eta\in C^\infty_c(A)$ denotes an arbitrary cut-off function.
  \end{itemize}
  Assume that $\nabla u_0$ is uniformly close to $\SO d$ in the sense that
  \begin{equation}\label{c:ass:u0sod}
    \|\dist(\nabla u_0,\SO d)\|_{L^\infty(A)}\leq \bar\rho,
  \end{equation}
  then
  \begin{eqnarray*}
     & &|\mathcal E_\e(v_\e) - \mathcal E_\e(u_\e)|+\|u_\e-v_\e\|_{H^1(A)}^2\\
     &\lesssim& C(\bar \rho)\left(\e^2\|\nabla^2u_0\|_{L^2(A)}^2+\int_A\left(\e^2|\nabla\eta|^2+(\eta-1)^2\right)\dist^2(\nabla u_0(x),\SO d)\,dx\right).
    \end{eqnarray*}
\end{proposition}
The result above (in particular assumption \eqref{c:ass:u0sod}) is tailor-made for the application in the proof of Theorem~\ref{T:NC}. Although Proposition~\ref{P:C} treats the simpler convex case, it seems to be new. An estimate on the two-scale expansion for a scalar monotone equation has been obtained in \cite{CPZ05}; their argument uses De-Giorgi-Nash-Moser regularity, and thus does not extend to the vectorial case.

The rest of this section is devoted to the proof of the proposition.

\begin{proof}[Proof of Proposition~\ref{P:C}]

\step 1 Chain rule 

We claim that $\phi(\tfrac{\cdot}\e,\nabla u_0)\in W^{1,r}(A,\R^d)$ and $\sigma(\tfrac{\cdot}\e,\nabla u_0)\in W^{1,r}(A,\R^{d\times d\times d})$ satisfy for almost every $x\in A$
\begin{equation}\label{chainrule}
\begin{split}
 &\nabla \phi(\tfrac{x}\e,\nabla u_0(x))=\e^{-1}\nabla_y\phi(\tfrac{x}\e,\nabla u_0(x))+ D\phi(\tfrac{x}\e,\nabla u_0(x))[\partial_j \nabla u_0(x)]\otimes e_j,\\
 &\nabla \sigma_{ijk}(\tfrac{x}\e,\nabla u_0(x))=\e^{-1}\nabla_y\sigma_{ijk}(\tfrac{x}\e,\nabla u_0(x))+ D\sigma_{ijk}(\tfrac{x}\e, \nabla u_0(x))[\partial_j \nabla u_0(x)]e_j^T,\quad\mbox{}
\end{split}
\end{equation}
for $i,j,k\in\{1,\dots,d\}$. We provide the argument only for $\phi$, since $\sigma$ can be treated in the same way. 

Sobolev embedding implies that $F\mapsto \phi(F)$ is in $C^1(U_{\bar{\bar\rho}},C_{\per,0}^{1,1-\frac{d}q}(Y))$ and thus 
\begin{align*}
 &y\mapsto \phi(y,F),~y\mapsto \nabla\phi(y,F),~y\mapsto D\phi(y,F)\quad\mbox{are measurable for every $F\in U_{\bar{\bar\rho}}$ and }\\
 &F\mapsto \phi(y,F),~F\mapsto \nabla \phi(y,F)\quad\mbox{are in $C^1(U_{\bar{\bar\rho}})$ and }F\mapsto D\phi(y,F)\mbox{ is in $C^{0,1-\frac{d}q}(U_{\bar{\bar\rho}})$ for every $y\in Y$}.
\end{align*}
Hence, \eqref{c:ass:u0sod} and $\bar\rho<\bar{\bar\rho}$ imply the measurability of $\phi(\frac{\cdot}\e,\nabla u_0)$ and of all functions in the first line of \eqref{chainrule}. Furthermore, the right-hand side in \eqref{chainrule} is in $L^r(A)$. Hence, we only need to show that \eqref{chainrule} holds in a distributional sense.

From $F\mapsto \phi(F)$ in $C^1(U_{\bar{\bar\rho}},C^{1,1-\frac{d}{q}}_\per(Y,\R^d))$ it follows that the map $(y,F)\mapsto \phi(y,F)$ is in $C^1(\R^d\times U_{\bar{\bar\rho}})$. Indeed, the existence of all partial derivatives is clear. Moreover, for all $F,G\in \R^{d\times d}$, $x,y\in\R^d$ it holds
%
\begin{align*}
 |\partial_{y_i}\phi(x,G)-\partial_{y_i}\phi(y,F)|\leq& \|\phi(G)-\phi(F)\|_{C^1(Y)}+|x-y|^{1-\frac{d}{q}}\|\phi(F)\|_{C^{1,1-\frac{d}{q}}(Y)}\\
 |\partial_{F_{ij}}\phi(x,G)-\partial_{F_{ij}}\phi(y,F)|\leq&\|D\phi(G)-D\phi(F)\|_{C^0(Y)}+|x-y|^{1-\frac{d}{q}}\|D\phi(F)\|_{C^{0,1-\frac{d}{q}}(Y)}, 
\end{align*}
for all $i,j=1,\dots,d$. Clearly, the right-hand side tends to zero whenever $(x,G)\to (y,F)$ and thus the continuity of the partial derivatives follows.

Using the $C^1$ regularity of $(y,F)\mapsto \phi(y,F)$, a proof of the chain rule \eqref{chainrule} can be found e.g. in \cite{Valent}. For the convenience of the reader we repeat the argument here: We choose a sequence of smooth functions $u_0^\rho\in C^\infty(A,\R^{d})$ satisfying $u_0^\rho\to u_0$ in $W^{2,r}(A)$ as $\rho\to0$ (and thus also in $W^{1,\infty}(A)$). Since $\nabla u_0(x)\in U_{\bar{\bar\rho}}$ for all $x\in \bar A$, we have for $\rho>0$ sufficiently small that $\nabla u_0^\rho(x)\in U_{\bar{\bar\rho}}$ for all $x\in \bar A$. Hence, for all $j=1,\dots,d$ and all $\varphi\in C_c^\infty(A)$, it holds
\begin{align*}
& \int_A \left(\e^{-1}(\partial_{y_j}\phi)(\tfrac{x}\e,\nabla u_0)+ D\phi(\tfrac{x}\e,\nabla u_0)[\partial_j \nabla u_0(x)]\right)\cdot\varphi\,dx+\int_A\phi(\tfrac{x}\e,\nabla u_0)\cdot\partial_j\varphi\,dx\\
&=\lim_{\rho\to0}\int_A \left(\e^{-1}(\partial_{y_j}\phi)(\tfrac{x}\e,\nabla u_0^\rho)+ D\phi(\tfrac{x}\e,\nabla u_0^\rho)[\partial_j \nabla u_0^\rho]\right)\cdot\varphi\,dx+\int_A\phi(\tfrac{x}\e,\nabla u_0^\rho)\cdot\partial_j\varphi\,dx=0,
\end{align*}
and thus the chain rule \eqref{chainrule} for $\phi$.

\step 2 $H^1$-estimate.

In this first step we prove an estimate that exploits the variational (and convex) structure of the problem: We claim that
\begin{equation*}
 \int_A|\nabla v_\e-\nabla u_\e|^2\,dx\lesssim \Lambda_\e,
\end{equation*}
where
\begin{eqnarray*}
 \Lambda_\e&:=& \e^2\int_A\left(|D\phi(\tfrac{x}{\e},\nabla u_0)|+|D\sigma(\tfrac{x}{\e},\nabla u_0)|\right)^2|\nabla^2u_0|^2\,dx\\
        &&+ \e^2\int_A|\nabla\eta|^2|\phi(\tfrac{x}{\e},\nabla u_0(x))|^2\,dx+                \int_A(\eta-1)^2|\nabla\phi(\tfrac{x}{\e},\nabla u_0(x))|^2.
\end{eqnarray*}
Indeed, since $u_\e$ is a minimizer and $V$ is strongly $\beta$-convex (which we exploit in form of \eqref{ineq:strong-conv2}), we have
\begin{eqnarray}\label{CC:eq:1}
 \frac{\beta}{2}\|\nabla v_\e-\nabla u_\e\|_{L^2(A)}^2&\leq& \mathcal E_\e(v_\e)-\mathcal E_\e(u_\e) \notag\\
  &\leq&\int_{A}DV(\tfrac{x}{\e},\nabla v_\e)[\nabla v_\e   -\nabla u_\e]-f\cdot(v_\e-u_\e)\,dx \notag\\
  &=&\int_{A}\Big(DV(\tfrac{x}{\e},\nabla v_\e)-DV_{\hom}(\nabla u_0)\Big)[\nabla v_\e-\nabla u_\e]\,dx,
\end{eqnarray}
where the last identity holds thanks to the Euler-Lagrange equation for $u_0$, i.e. $-\divv(DV_{\hom}(\nabla u_0))=f$ (in its weak form). We rewrite the flux difference on the right-hand side of \eqref{CC:eq:1} as
\begin{eqnarray*}
 DV(\tfrac{x}{\e},\nabla v_\e(x))-DV_{\hom}(\nabla u_0(x))&=&J(\tfrac{x}{\e},\nabla u_0(x))+R_\e(x)\qquad\text{where}\\
 J(y,F)&=&DV(y,F+\nabla\phi(y,F))-DV_{\hom}(F)\qquad\mbox{see \eqref{def:J}}\\
 R_\e(x)&:=&DV(\tfrac{x}{\e},\nabla v_\e(x))-DV\big(\tfrac{x}{\e},\nabla u_0(x)+\nabla\phi(\tfrac{x}{\e},\nabla u_0(x))\big).
\end{eqnarray*}
Hence, it suffices to prove the following two estimates
\begin{eqnarray}
 \left|\int_{A} \langle J(\tfrac{x}{\e},\nabla u_0(x)),\nabla v_\e-\nabla u_\e \rangle\,dx\right|&\lesssim& \Lambda_\e^\frac 12\|\nabla v_\e-\nabla u_\e\|_{L^2(A)},\label{CC:eq:2}\\
 \int_{A}|R_\e|^2\,dx&\lesssim& \Lambda_\e.\label{CC:eq:3}
\end{eqnarray}

\textit{Substep 1.1} (Estimate of \eqref{CC:eq:2}).

Recall that
\begin{equation*}
 -\sum_{k=1}^d\partial_{k}\sigma_{ijk}(y,F)=J_{ij}(y,F)\quad\mbox{for $i,j=1,\dots,d$ and all $y\in\R^d$,}
\end{equation*}
see \eqref{prop:sigma2}, and thus 
\begin{eqnarray*}
 \e\sum_{k=1}^d\partial_k\sigma_{ijk}(\tfrac{x}{\e},\nabla u_0(x))&=&\sum_{k=1}^d\left((\partial_k \sigma_{ijk})(\tfrac{x}{\e},\nabla u_0(x))+\e D\sigma_{ijk}(\tfrac{x}{\e},\nabla u_0(x))[\partial_k\nabla u_0(x)]\right)\\
 &=&-J_{ij}(\tfrac{x}{\e},\nabla u_0(x))+\e \sum_{k=1}^dD\sigma_{ijk}(\tfrac{x}{\e},\nabla u_0(x))[\partial_k\nabla u_0(x)].
  \end{eqnarray*}
Hence, for all $\eta\in C^\infty_0(A)$, we have (by an integration by parts)
\begin{eqnarray*}
 \int_A \langle J(\tfrac{x}{\e},\nabla u_0(x)),\nabla \eta\rangle&=&\e\sum_{i,j,k=1}^d\int_{A}D\sigma_{ijk}(\tfrac{x}{\e},\nabla u_0(x))[\partial_k\nabla u_0(x)] \partial_j(\eta\cdot e_i)\,dx\\
    & &\quad+\e\sum_{i,j,k=1}^d\int_A\sigma_{ijk}(\tfrac{x}{\e},\nabla u_0(x))\partial_k\partial_{j}(\eta\cdot e_i)\,dx,
  \end{eqnarray*}
where the last term on the right-hand side vanishes, since $\sigma_{ijk}=-\sigma_{ikj}$ and $\nabla^2 \eta$ is symmetric. By density of $C^\infty_0(A)$ in $H^1_0(A)$ this identity holds for $\eta=v_\e-u_\e$. Hence,
\begin{equation*}
 \left|\int_A \langle J(\tfrac{x}{\e},\nabla u_0(x)) , \nabla v_\e-\nabla u_\e\rangle \right|\leq 
 \e\sum_{k=1}^d\|D\sigma(\tfrac{\cdot}{\e},\nabla u_0)[\partial_k\nabla u_0]\|_{L^2(A)}\|\nabla v_\e-\nabla u_\e\|_{L^2(A)},
\end{equation*}
and thus \eqref{CC:eq:2}.

\textit{Substep 1.2} (Estimate of \eqref{CC:eq:3}).

A direct calculation shows that
\begin{eqnarray*}
 \nabla v_\e(x)&=&\nabla u_0(x)+\nabla_y\phi(\tfrac{x}{\e},\nabla u_0(x)) + G_\e(x)\qquad\text{where }\\
   G_\e(x)&:=&(\eta - 1)\nabla_y\phi(\tfrac{x}{\e},\nabla u_0(x))+\e\eta D\phi(\tfrac{x}{\e},\nabla u_0(x))[\partial_k\nabla u(x)]\otimes e_k+\e \phi(\tfrac{x}{\e},\nabla u_0(x))\otimes\nabla\eta.
\end{eqnarray*}
Note that $\|G_\e\|_{L^2(A)}^2\leq \Lambda_\e$.  Since the fourth order tensor $D^2V(y,F)$ is bounded (uniformly in $y$ and $F$), we deduce that
\begin{eqnarray*}
 |R_\e(x)|\leq \left(\int_0^1\left|D^2V\left(\tfrac{x}{\e},\nabla u_0(x)+\nabla_y\phi(\tfrac{x}{\e},\nabla u_0(x))+tG_\e(x)\right)\right|\,dt\right)|G_\e(x)|\lesssim |G_\e(x)|,
\end{eqnarray*}
and thus \eqref{CC:eq:3}.

\step 3 Conclusion

In this step we improve the estimate of Step~2 by appealing to the Lipschitz regularity of the extended corrector $(\phi,\sigma)$. By Lemma~\ref{L:convexcorrector} the maps $F\mapsto \phi(F)$ and $F\mapsto \sigma(F)$ are locally Lipschitz from $U_{\bar{\bar\rho}}$ to $W^{1,\infty}(A)$. Hence, using \eqref{c:ass:u0sod} and $\bar\rho<\bar{\bar\rho}$, we obtain 
  $$\|D\phi(\nabla u_0(x))\|_{W^{1,\infty}(Y)}+\|D\sigma(\nabla u_0(x))\|_{W^{1,\infty}(Y)}\lesssim C(\bar\rho)\quad\mbox{for every $x\in A$.}$$
  Moreover, since $\phi(R)=0$ for every $R\in \SO d$, we obtain
  $$\|\phi(\nabla u(x))\|_{W^{1,\infty}(Y)}\lesssim C(\bar\rho)\dist(\nabla u_0(x),\SO d).$$
  Combining the previous two displayed formulas with the definition of $\Lambda_\e$, we obtain 
  \begin{equation*}
    \Lambda_\e\lesssim C(\bar\rho)\left(\e^2\int_A|\nabla^2u_0|^2+\int_A\left(\e^2|\nabla\eta|^2+(\eta-1)^2\right)\dist^2(\nabla u_0(x),\SO d)\,dx\right).
  \end{equation*}
  By Poincar\'e's inequality this proves the claimed estimate for $\|\nabla v_\e-\nabla u_\e\|_{H^1(A)}$. The corresponding estimate for the energy difference $\mathcal E_\e(v_\e)-\mathcal E_\e(u_\e)$ is a direct consequence of the minimality of $u_\e$ and the boundedness of $D^2V$:
  \begin{align}\label{est:diffene}
   \mathcal E_\e(v_\e)-\mathcal E_\e(u_\e)&=\int_A V(\tfrac{x}\e,\nabla v_\e)-V(\tfrac{x}\e,\nabla u_\e)-f\cdot (v_\e-u_\e)\,dx\notag\\
   &=\int_A\int_0^1(1-s)D^2V(\tfrac{x}\e,\nabla u_\e+s(\nabla v_\e-\nabla u_\e))[\nabla v_\e-\nabla u_\e,\nabla v_\e-\nabla u_\e]\,ds\,dx\notag\\
   &\leq \frac{1}{2\beta}\int_A |\nabla v_\e-\nabla u_\e|^2,
  \end{align}
  which finishes the proof.
  
\end{proof}

\subsection{Proof of Theorem~\ref{T:NC}}\label{S:NC}

The proof is structured as follows. We first prove that the minimizer $u_0$ of $\mathcal E_{\hom}$ satisfies the nonlinear estimates \eqref{NC:NL:1} and \eqref{NC:NL:2}. This will be done in Step~1 by appealing to the implicit function theorem, which is one of the reasons why we require $\bar\rho$ (and thus the data) to be small. In Step~2 we prove that $u_0$ is also the unique minimizer of the non-convex minimization problem. This is a consequence of the identity
\begin{equation*}
  W_{\ho}(\nabla u_0)=V_{\ho}(\nabla u_0)-\mu\det\nabla u_0,
\end{equation*}
which is available, provided $\nabla u_0$ is sufficiently close to $\SO d$ (uniformly in $x$), see \eqref{NC:NL:1} and \eqref{pf:onecell}. The core of the proof is the argument for part (c). Here we modify the two-scale expansion $v_\e$ of Proposition~\ref{P:C} in such a way that $\nabla v_\e$ is sufficiently close to $\SO d$, so that we can exploit the matching property \eqref{W=V}.

\medskip

\subsubsection*{\bf Proof of (a) and (b).}

\step 1 (Minimizer for $\mathcal E_{\hom}$ and nonlinear estimates.)

Given $g\in H^1(A)$, $f\in L^2(A)$, we denote by $w(f,g)\in g+H_0^1(A)$ the unique minimizer of the functional \eqref{ene:convexhom}. Since $V_\ho$ is sufficiently smooth, $w(f,g)$ is characterized as the unique solution $\varphi\in g+H_0^1(A)$ of the corresponding Euler-Lagrange equation
\begin{equation*}
 \int_A D V_\ho(\nabla \varphi)[\nabla \eta]=\int_A f\cdot\eta\quad\mbox{for all $\eta\in H_0^1(A)$}.
\end{equation*}
Note that $w(f,g)=g_0+h+\varphi$ where $h=g-g_0$ and $\varphi\in H^1_0(A)$ minimizes
\begin{equation*}
 \varphi\mapsto \int_A V_\ho(\nabla g_0+\nabla h+\nabla\varphi)-f\cdot(g_0+h+\varphi).
\end{equation*}
For $\rho>0$ set
\begin{align*}
 &B_1(\rho):=\{\,(f,h)\in L^r(A)\times W^{2,r}(A)\,|\, \|f\|_{L^r(A)}+\|h\|_{W^{2,r}(A)}<\rho\,\}\\
 &B_2(\rho):=\{\,\varphi\in H_0^1(A)\cap W^{2,r}(A)\,|\, \|\varphi\|_{W^{2,r}(A)}<\rho\,\}.
\end{align*}
By Sobolev embedding, we find $\rho'\in(0,\bar{\bar\rho})$ such that for all $(f,h)\in B_1(\rho')$ and $\varphi\in B_2(\rho')$
\begin{equation}\label{rel:g0}
 \|\dist(\nabla g_0,\SO d)\|_{L^\infty(A)}<\rho'\quad\mbox{implies}\quad\|\dist(\nabla g_0+\nabla h+\nabla\varphi,\SO d)\|_{L^\infty(A)}<\bar{\bar\rho}.
\end{equation}
For given $g_0\in W^{2,r}(A)$ satisfying  $\|\dist(\nabla g_0,\SO d)\|_{L^\infty(A)}<\rho'$, we define the mapping  $T:B_1(\rho')\times B_2(\rho')\to L^r(A)$ given by
\begin{equation*}
 T(f,h,\varphi):=-\divv\big(DV_\ho(\nabla g_0+\nabla h+\nabla\varphi)\big)-f.
\end{equation*}
In analogy to the proof of Lemma~\ref{L:convexcorrector}, we have
\begin{enumerate}[(a)]
 \item $T\in C^1(B_1(\rho')\times B_2(\rho'), L^r(A))$\\
 \textit{Argument:} This is a direct consequence of \eqref{rel:g0}, $V_\ho\in C^3(U_{\bar{\bar\rho}})$ and the differentiability of the corresponding composition operators, cf.~Lemma~\ref{L:composition}. In particular, we have
 $$D_\varphi T(f,g,\varphi)[\phi]=-\divv \left(D^2V_\ho(\nabla g_0+\nabla h+\nabla \varphi)[\nabla\phi]\right).$$
 \item $T(0,0,0)=0$\\
 \textit{Argument:} This follows by \eqref{NC:lin_bc} and $V_\ho=W_\ho+\mu \det$ on $U_{\bar{\bar\rho}}$.
 \item $D_\varphi T(0,0,0):H_0^1(A)\cap W^{2,r}(A)\to L^r(A)$ is invertible.\\
 \textit{Argument:} By (a), we have that 
 $$D_\varphi T(0,0,0)[\phi]=-\divv \left(D^2V_\ho(\nabla g_0)[\nabla\phi]\right).$$
 The fourth order tensor $D^2V_\ho$ is uniformly elliptic and a combination of $V_\ho\in C^3(U_{\bar{\bar \rho}})$, \eqref{rel:g0} and $\nabla g_0\in W^{1,r}(\R^d)$ yields $\mathbb L:=D^2V_\ho(\nabla g_0)\in W_\loc^{1,r}(\R^d)\hookrightarrow C_\loc^{0,1-\frac{d}r}(\R^d)$. Hence, we obtain the invertibility of $D_\varphi T(0,0,0)$ by maximal $L^r$-regularity. 
 
For convenience of the reader, we recall here an argument for the validity of maximal regularity in this case: Fix $f\in L^r(A)$, $r>d$. By maximal regularity and $\mathbb L\in C(\bar A)$ , every weak solution $\varphi\in H_0^1(A)$ of
$$-\divv \mathbb L\nabla \varphi:=-\partial_\alpha \mathbb L_{\alpha \beta}^{ij}\partial_\beta  \varphi^j=f^i$$
satisfies $\varphi\in W^{1,s}(A)$ for every $s<\infty$. Moreover, appealing to maximal regularity for systems in non-divergence form, $\nabla \varphi\in L^s(A)$ for every $s<\infty$, $\mathbb L\in W^{1,r}(A)$ with $r>d$, and 
$$-\mathbb L_{\alpha\beta}^{ij}\partial_\alpha\partial_\beta \varphi=f^i+(\partial_{\alpha} \mathbb L_{\alpha \beta}^{ij})\partial_\beta  \varphi^j,$$
we obtain firstly that $\varphi\in W^{2,\tilde r}(A)$ for every $\tilde r<r$ which implies $\nabla \varphi\in L^\infty(A)$ and thus eventually $\varphi\in W^{2,r}(A)$.
\end{enumerate}
Hence, by the implicit function theorem, there exists $\rho''>0$ and a map $\Lambda\in C^1(B_2(\rho''), H_0^1(A)\cap W^{2,r}(A))$ such that $T(f,h,\Lambda(f,h))=0$, $\Lambda(f,h)\in B_2(\rho')$ and
\begin{eqnarray*}
  \|\Lambda(f,h)\|_{W^{2,r}(\R^d)}&\lesssim&C(\rho'')\left(\|f\|_{L^r(A)}+\|h\|_{W^{2,r}(A)}\right).
\end{eqnarray*}
Hence, $w:=g_0+h+\Lambda(h,f)=g+\Lambda(h,f)\in H_g^1(A)\cap W^{2,r}(A)$ satisfies $\dist(\nabla w,\SO d)<\bar{\bar \rho}$, and
\begin{equation*}
 \int DV_\ho(\nabla w))[\nabla \eta]-f\cdot\eta=0\qquad\text{ for all }\eta\in H^1_0(A),
\end{equation*}
and
\begin{align*}
 &\|\dist(\nabla w, \SO d)\|_{L^\infty(A)}+\|w-g_0\|_{W^{2,r}(A)},\\
 &\lesssim C(\rho'')\left(\|f\|_{L^r(A)}+\|g-g_0\|_{W^{2,r}(A)}+\|\dist(\nabla g_0,\SO d)\|_{L^\infty(A)}\right).
\end{align*}

\step 2 Argument for (a).

Next, we show that $w(f,g)$ is the unique minimizer in $g+W_0^{1,p}(A)$ of \eqref{ene:nonconvexhom}. Given $v\in g+W_0^{1,p}(A)$, it holds 
\begin{align}\label{est:lminhom}
  & \int_A W_\ho(\nabla v)-f\cdot v-(W_\ho(\nabla w(f,g))-f\cdot w(f,g))\notag\\
  &\, \stackrel{\eqref{WhgeqVh},\eqref{pf:onecell}}{\geq} \int_A V_\ho(\nabla v)-f\cdot v-\mu\det(\nabla v)-V_\ho(\nabla w(f,g))+f\cdot w(f,g)+\mu\det(\nabla w(f,g))\notag\\
  &\, = \int_A V_\ho(\nabla v)-f\cdot v-(V_\ho(\nabla w(f,g))-f\cdot w(f,g))\geq0,
\end{align}
where the determinant terms vanishes since $v-w(f,g)\in W_0^{1,p}(A)$ and $p\geq d$. Hence, $w(f,g)$ is a minimizer of $\mathcal I_\ho$ in $g+W_0^{1,p}(A)$. Moreover, the strong convexity of $V_\ho$ implies that the last inequality in \eqref{est:lminhom} is an equality if and only if $v\equiv w(f,g)$, and thus uniqueness follows.

\medskip

\subsubsection*{\bf Proof of (c).}

In order to illustrate the idea of the proof, we first give a heuristic argument for (c) based on the assumption that
\begin{equation}\label{eq:unrealistic}
  \|\dist(\nabla v_\e,\SO d)\|_{L^\infty(A)}<\delta,
\end{equation}
where $v_\e$ is a two-scale expansion with cut-off as in Proposition~\ref{P:C}. Based on this assumption we show that
\begin{equation}\label{NC:master:ineq}
  \frac{\beta}{4}\int_A|\nabla w_\e-\nabla v_\e|^2\leq 2(\mathcal E_\e(v_\e)-\mathcal E_\e(u_\e))+\left(\mathcal I_\e(w_\e)-\inf_{g+W_0^{1,p}(A)}\mathcal I_\e\right).
\end{equation}
In view of the convex error estimate, see Proposition~\ref{P:C}, this implies the estimate of (c) in a form where $v_\e$ replaces the two-scale expansion without cut-off. \eqref{NC:master:ineq} can be seen as follows: By strong $\beta$-convexity of $V$, and the fact that $u_\e$ minimizes $\mathcal E_\e$, we have
\begin{equation*}
  \frac{\beta}{2}\int_{A}|\nabla w-\nabla u_\e|^2\leq \mathcal E_\e(w)-\mathcal E_\e(u_\e)\qquad\text{for all }w\in H^1_g(A).
\end{equation*}
We specify this estimate to $w=w_\e$ and $w=v_\e$ and deduce that
\begin{eqnarray*}
  \frac{\beta}{4}\int_{A}|\nabla w_\e-\nabla v_\e|^2&\leq&   \frac{\beta}{2}\int_{A}|\nabla w_\e-\nabla u_\e|^2+  \frac{\beta}{2}\int_{A}|\nabla u_\e-\nabla v_\e|^2\\
  &\leq&\mathcal E_\e(w_\e)-\mathcal E_\e(u_\e)+\mathcal E_\e(v_\e)-\mathcal E_\e(u_\e)\\
  &\leq&\mathcal E_\e(w_\e)-\mathcal E_\e(v_\e)+2\left(\mathcal E_\e(v_\e)-\mathcal E_\e(u_\e)\right).
\end{eqnarray*}
We estimate the first term on the right-hand side by appealing to $V(\cdot,F)\leq W(\cdot,F)+\mu \det(F)$ and the fact that equality holds for $F\in U_\delta$, that is assumption \eqref{eq:unrealistic}. We get
\begin{eqnarray*}
  \mathcal E_\e(w_\e)-\mathcal E_\e(v_\e)\leq \mathcal I_\e(w_\e)-\mathcal I_\e(v_\e)+\mu \int_A\det(\nabla w_\e)-\det(\nabla v_\e)\,dx\leq \mathcal I_\e(w_\e)-\inf_{g+W_0^{1,p}(A)}\mathcal I_\e,
\end{eqnarray*}
where the integral that invokes the determinant vanishes, since $w_\e-v_\e\in W^{1,p}_0(A)$ with $p\geq d$. This completes the argument for \eqref{NC:master:ineq} and thus the heuristic argument for (c). 
\medskip

In order to turn this argument into a rigorous proof, we need to circumvent assumption \eqref{eq:unrealistic}. Note that if $\nabla^2 u_0\in L^\infty(A)$, \eqref{eq:unrealistic} follows for sufficiently small loads and sufficiently small $\e\ll 1$, see below. However, in general $\nabla^2u_0$ is not bounded. We therefore replace $\nabla^2 u_0$ in the definition of $v_\e$ by a truncated version. We denote the resulting two-scale expansion by $\bar v_\e$. We are going to show that $\nabla v_\e-\nabla\bar v_\e$ is close and that \eqref{eq:unrealistic} holds for $\nabla\bar v_\e$ for all $\e\leq\e_0$, where $\e_0$ can be chosen only depending on $\bar\rho$ and the size of $\|\e\nabla\eta\|_{L^\infty(A)}$ and $\|\nabla^2u_0\|_{L^r(A)}$. In order to get a uniform estimate, we fix the cut-off function from now on. To that end note that, since $A$ is Lipschitz, there exists a constant $c=c(A)>0$ such that for any $\e\in(0,1)$ there exists a cut-off function $\eta$ (depending on $\e$) such that
\begin{equation}\label{eq:cut-off}
  \int_A(\e^2|\nabla\eta|^2+(1-\eta)^2)\leq c(A)\e,\qquad \e\|\nabla\eta\|_{L^\infty(A)}\leq 1.
\end{equation}
From now on we assume that $\eta$ satisfies the conditions above.

\step 1 Lipschitz truncation.

We claim that for every $\e\in(0,1)$ and $\bar\rho>0$ there exists a vector-field $(\nabla u_0)_\e\in W^{1,\infty}(A,\R^{d\times d})$ satisfying for $i=1,\dots,d$
\begin{subequations}
 \begin{eqnarray}
  \| \partial_i(\nabla u_0)_\e\|_{L^\infty(A)}&\lesssim&   \e^{-\frac{d}{r}}(\|\nabla^2g_0\|_{L^r(A)}+\bar\rho),\label{u0:trunc:1}\\
  \| \partial_i(\nabla u_0)_\e - \partial_i\nabla u_0\|_{L^s(A)} &\lesssim& \e^{\frac{d}s-\frac{d}{r}}(\|\nabla^2g_0\|_{L^r(A)}+\Lambda(f,g,g_0))\quad\mbox{for all $s\in[1,r]$},\label{u0:trunc:2}\\
  \|(\nabla u_0)_\e - \nabla u_0\|_{L^s(A)}&\lesssim& \e^{1+\frac{d}s-\frac{d}{r}}(\|\nabla^2g_0\|_{L^r(A)}+\Lambda(f,g,g_0))\quad\mbox{for all $s\in[1,\infty]$},\label{u0:trunc:3}\\
  |\{ (\nabla u_0)_\e \neq  \nabla u_0\}|&\lesssim&\e^d,\label{u0:trunc:4}
 \end{eqnarray}
\end{subequations}
where we set $\frac1\infty=0$. It suffices to argue that there exists $C=C(A,d,r)$ such that for every $\lambda>1$ there exists a vector-field $(\nabla u_0)_\lambda\in W^{1,\infty}(A,\R^{d\times d})$ satisfying for $i=1,\dots,d$
\begin{subequations}
 \begin{eqnarray}
  \| \partial_i(\nabla u_0)_\lambda\|_{L^\infty(A)}&\leq& C\lambda,\label{u0:trunc:1b}\\
  \| \partial_i(\nabla u_0)_\lambda - \partial_i\nabla u_0\|_{L^s(A)} &\leq& C \lambda^{1-\frac{r}s}\|\nabla^2 u_0\|_{L^r(A)}^\frac{r}s\quad\mbox{for all $s\in[1,r]$},\label{u0:trunc:2b}\\
  \|(\nabla u_0)_\lambda - \nabla u_0\|_{L^s(A)}&\leq& C\lambda^{1-\frac{r}{d}-\frac{r}{s}}\|\nabla^2 u_0\|_{L^r(A)}^{\frac{r}{s}+\frac{r}{d}}\quad\mbox{for all $s\in[1,\infty]$},\label{u0:trunc:3b}\\
  |\{ (\nabla u_0)_\lambda \neq  \nabla u_0\}|&\leq& C\lambda^{-r}\|\nabla^2 u_0\|_{L^{r}(A)}^r.\label{u0:trunc:4b}
 \end{eqnarray}
\end{subequations}
Indeed, the claim then follows by choosing $\lambda=\e^{-\frac{d}{r}}(\|\nabla^2g_0\|_{L^r(A)}+\bar\rho)$ and using 
$$\|\nabla^2 u_0\|_{L^r(A)}\leq \|\nabla^2 g_0\|_{L^r(A)}+\Lambda(f,g,g_0)\leq \|\nabla^2 g_0\|_{L^r(A)}+\bar\rho.$$ 
In the construction of $(\nabla u_0)_{\lambda}$ we distinguish the cases $\lambda\geq \lambda_0$ and $0<\lambda<\lambda_0$, with threshold $\lambda_0$ determined below, see \eqref{est:trunc:1}.

\substep{1.1} Estimates \eqref{u0:trunc:1b}--\eqref{u0:trunc:4b} for $\lambda$ sufficiently large.

By Lipschitz truncation in the version of \cite[Proposition A.1]{FJM02}, there exists a constant $C'=C'(A,d,r)>1$ such that for every $\lambda>0$ we find a vector-field $(\nabla u_0)_\lambda\in W^{1,\infty}(A,\R^{d\times d})$ satisfying
\begin{eqnarray*}
 \|\partial_i(\nabla u_0)_\lambda\|_{L^{\infty}(A)}&\leq& C' \lambda\quad\mbox{for $i=1,\dots,d$,}\\
 |E_\lambda|:=|\{ (\nabla u_0)_\lambda \neq  \nabla u_0\}|&\leq& C'\lambda^{-r}\|\nabla^2 u_0\|_{L^{r}(A)}^r.
\end{eqnarray*}
Thus, $(\nabla u_0)_\lambda$ already satisfies \eqref{u0:trunc:1b} and \eqref{u0:trunc:4b}, and it remains to prove \eqref{u0:trunc:2b} and \eqref{u0:trunc:3b}. Note that the combination of the two estimates yields for every $s\in[1,r]$ and $i=1,\dots,d$:
\begin{eqnarray*}
 \|\partial_i (\nabla u_0)_\lambda- \partial_i \nabla u_0\|_{L^s(A)}&\leq& \left(\int_{\{\nabla u_0^\lambda\neq \nabla u_0\}}|\partial_i(\nabla u_0)_\lambda|^s\right)^{\frac1s}+\left(\int_{\{\nabla u_0^\lambda\neq \nabla u_0\}}|\nabla^2 u_0|^s\right)^{\frac1s}\\
  &\leq& C'\lambda|E_\lambda|^\frac1s+|E_\lambda|^{\frac1s-\frac1r}\|\nabla^2 u_0\|_{L^r(A)}\\
&\leq&2C'\lambda^{1-\frac{r}s}\|\nabla^2 u_0\|_{L^r(A)}^\frac{r}s,
\end{eqnarray*}
and thus \eqref{u0:trunc:2b}. Next we estimate $(\nabla u_0)_{\lambda}-\nabla u_0$. We claim that there exists $C''=C''(A,d,r)$ such that
\begin{equation}\label{est:trunc:1}
 \|\nabla u_0^\lambda - \nabla u_0\|_{L^\infty(A)}\leq C(A,d,r)\lambda^{1-\frac{r}{d}}\|\nabla^2 u_0\|_{L^r(A)}^\frac{r}{d}\qquad\mbox{for all $\lambda>\lambda_0=C''\|\nabla u_0\|_{L^r(A)}$}.
\end{equation} 
Notice that \eqref{est:trunc:1} implies \eqref{u0:trunc:3b} for $\lambda\geq\lambda_0$ by the following calculation:
\begin{align*}
 \|(\nabla u_0)_\lambda - \nabla u_0\|_{L^s(A)}\leq& \|(\nabla u_0)_\lambda - \nabla u_0\|_{L^\infty(A)} |E_\lambda|^\frac1s\leq C(A,d,r)\lambda^{1-\frac{r}{d}-\frac{r}s}\|\nabla^2 u_0\|_{L^r(A)}^{\frac{r}{d}+\frac{r}{s}}.
\end{align*}
To show \eqref{est:trunc:1}, we use that the Lipschitz regularity of $A$ implies that $A$ satisfies the following cone condition: There exists a cone $\Co\subset\R^d$ with height $h>0$, aperture angle $\kappa\in(0,\pi]$ and vertex in $0$, such that
$$x\in A\Rightarrow x+T_x(\Co)\subset A$$
where $T_x$ is a rigid motion, cf.~\cite{Adams}~p.84. Set $C''=(C'|\operatorname{Co}|^{-1})^\frac{1}{r}$. Then, for every $\lambda\geq\lambda_0:=C''\|\nabla^2u_0\|_{L^r(A)}$ it holds  
\begin{align*}
 |E_\lambda|\leq C'\lambda^{-r}\|\nabla^2 u_0\|_{L^r(A)}^r\leq|\Co|.
\end{align*}
Using that every cone $\operatorname{Co}'$ with height $h$ and fixed aperture angle $\kappa$ satisfies $|\operatorname{Co}'|=C(d,\kappa)h^d$, we find for every $x\in E_\lambda$ with $\lambda > \lambda_0$ a point $y\in x+ T_x(\operatorname{Co})\subset A$ satisfying
$$|x-y|<C(A,d,r)\lambda^{-\frac{r}d}\|\nabla^2 u_0\|_{L^r(A)}^{\frac{r}d}\quad\mbox{and}\quad \nabla u_0^\lambda(y)= \nabla u_0(y).$$
Thus, appealing to the Lipschitz estimate for $(\nabla u_0)_\lambda$ and Morrey's inequality (on a cone with apeture angle $\kappa$ and height $|y-x|(\leq h)$, see Lemma~\ref{L:Morreycone})
, we obtain
\begin{align*}
 |(\nabla u_0)_{\lambda}(x)- \nabla u_0(x)|\leq& |(\nabla u_0)_{\lambda}(x)- \nabla u_0^\lambda(y)|+|\nabla u_0(y) - \nabla u_0(x)|\\
 \leq& C(A,d,r)(|x-y|\lambda + |x-y|^{1-\frac{d}{r}}\|\nabla^2 u_0\|_{L^r(A)})\\
 \leq& C(A,d,r)\lambda^{1-\frac{r}{d}}\|\nabla^2 u_0\|_{L^r(A)}^{\frac{r}{d}},
\end{align*}
and by the arbitrariness of $x\in E_\lambda$, we obtain \eqref{est:trunc:1}. 

\substep{1.2} Estimates \eqref{u0:trunc:1b}--\eqref{u0:trunc:4b} for $\lambda\in(1,\lambda_0]$.

We claim that for $\lambda\leq \lambda_0:=C''\|\nabla^2 u_0\|_{L^r(A)}$ the constant $(\nabla u_0)_\lambda=\fint_A \nabla u_0$ satisfies the sought after properties \eqref{u0:trunc:1b}--\eqref{u0:trunc:3b}. Indeed, \eqref{u0:trunc:1b} is trivial. H\"older's inequality and the definition of $\lambda_0$ yields for $s\in[1,r]$
\begin{align*}
 \|\partial_i\nabla u_0\|_{L^s(A)}\leq  |A|^{\frac1r-\frac1s} \|\nabla^2 u_0\|_{L^r(A)}^{1-\frac{r}{s}}\|\nabla^2 u_0\|_{L^r(A)}^\frac{r}{s}\leq C(A,d,r)\lambda_0^{1-\frac{r}{s}}\|\nabla^2 u_0\|_{L^r(A)}^\frac{r}{s}.
\end{align*}
Since $\frac{r}{s}\geq1$, this implies \eqref{u0:trunc:2b} for $\lambda\leq \lambda_0$. The estimate \eqref{u0:trunc:3b} follows similarly by appealing to Sobolev's inequality. Finally, note that \eqref{u0:trunc:4b} is a consequence of $C(A)|A|\lesssim C''\leq \lambda_0^{-r}\|\nabla^2 u_0\|_{L^r(A)}^r\leq \lambda^{-r}\|\nabla^2 u_0\|_{L^r(A)}^r$.

\step 2 Modified asymptotic expansion.

We claim that there exist $\bar \rho\lesssim1$ such that for all $\e\in (0,1)$ we find $\bar v_\e\in g+W_0^{1,\infty}(A)$ satisfying
\begin{eqnarray}
 \|\dist(\nabla \bar v_\e,\SO d)\|_{L^\infty(A)}&<&\delta,\label{est:barve1}\\
 \|v_\e-\bar v_\e\|_{H^1(A)}&\lesssim& C(\bar \rho)\e(1+\|\nabla^2 g_0\|_{L^r(A)}^\frac{r}{r-d}) (\Lambda(f,g,g_0)+\|\nabla^2 g_0\|_{L^r(A)})\label{est:barve2}.
\end{eqnarray}

\substep{2.1} Estimates for $\e\in(0,\e_0)$ with $\e_0>0$ sufficiently small.

We claim that there exists $\bar\rho\lesssim1$ and $\e_0\lesssim (1+\|\nabla^2 g_0\|_{L^r(A)})^{-\frac{r}{r-d}}$ sufficiently small such that for every $\e\in (0,\e_0)$ we find $\bar v_\e\in g+W_0^{1,\infty}(A)$ satisfying \eqref{est:barve1} and 
\begin{equation}\label{claim:vel:h1}
  \|\bar v_\e-v_\e\|_{H^1(A)}\lesssim C(\bar\rho) \e(\|\nabla^2 g_0\|_{L^{r}(A)}+\Lambda(f,g,g_0)).
\end{equation}
First, we note that
\begin{align*}
 &\|\dist((\nabla u_0)_\e,\SO d)\|_{L^\infty(A)}\\
 &\qquad\leq \|\dist(\nabla g_0,\SO d)\|_{L^\infty(A)}+\|\nabla g_0 - \nabla u_0\|_{L^\infty(A)}+\|\nabla u_0-(\nabla u_0)_\e\|_{L^\infty(A)}\\
 &\qquad\lesssim \bar \rho+\e^{1-\frac{d}{r}}(\bar\rho+\|\nabla^2 g_0\|_{L^r(A)})
\end{align*}
and thus for $\bar \rho\lesssim1$ and $\e_0\lesssim(1+\|\nabla^2 g_0\|_{L^r(A)})^{-\frac{r}{r-d}}$ sufficiently small
\begin{equation}
 \|\dist((\nabla u_0)_\e,\SO d)\|_{L^\infty(A)}<\bar{\bar \rho}\quad\mbox{for all $\e\in(0,\e_0)$}.\label{est:u0l:1}
\end{equation}
From \eqref{est:u0l:1}, we deduce that for all $\e\in(0,\e_0)$  
\begin{equation}\label{def:vepsl}
 \bar v_\e:=u_0+\e \eta \phi(\tfrac{\cdot}\e,(\nabla u_0)_\e)\quad\mbox{satisfies}\quad \bar v_\e\in W^{1,\infty}(A),
\end{equation}
and it holds for almost every $x\in A$
\begin{eqnarray*}
  \nabla \bar v_\e(x)&=&\nabla u_0(x)+\nabla\big(\e\eta(x)\phi(\tfrac{x}{\e},(\nabla u_0)_\e(x))\big)\\
  &=&\nabla u_0(x)+\eta(x)\nabla_y\phi(\tfrac{x}{\e},(\nabla u_0)_\e(x))+\e\sum_{j=1}^d\eta(D\phi(\tfrac{x}{\e},(\nabla u_0)_\e(x))[\partial_j(\nabla u_0)_\e(x)]\otimes e_j)\\
  &&+\e\phi(\tfrac{x}{\e},(\nabla u_0)_\e(x))\otimes\nabla\eta(x).     
\end{eqnarray*}
To simplify the notation, we drop the argument $\frac{x}\e$ in the following. The estimate \eqref{est:u0l:1} and $\phi\in C^1(U_{\bar{\bar\rho}},W_{\per,0}^{1,\infty}(Y))$ imply for almost every $x\in A$ that
$$\| \phi((\nabla u_0)_\e(x))-\phi(\nabla u_0(x))\|_{W^{1,\infty}(Y)}\lesssim|(\nabla u_0)_\e(x)-\nabla u_0(x)|\quad\mbox{and}\quad \|D\phi((\nabla u_0)_\e(x))\|_{L^\infty(Y)}\lesssim1.$$
Combining the above estimates with \eqref{eq:cut-off}, \eqref{u0:trunc:1}, and $\phi(R)=0$ for all for all $R\in \SO d$ (see Lemma~\ref{L:convexcorrector}), we obtain for almost every $x\in A$ that
\begin{align*}
 |\nabla \bar v_\e(x)-\nabla u_0(x)|\lesssim& (1+|\e\nabla \eta(x)|)\|\phi((\nabla u_0)_\e(x))\|_{W^{1,\infty}(Y)}+\e|\nabla (\nabla u_0)_\e(x)|\notag\\
 \lesssim&\|\phi(\nabla u_0(x))\|_{W^{1,\infty}(Y)}+|(\nabla u_0)_\e(x)-\nabla u_0(x)|+ \e^{1-\frac{d}{r}}(\|\nabla^2g_0\|_{L^r(A)}+\bar\rho)\notag\\
 \lesssim&C(\bar \rho)\left(\dist(\nabla u_0(x),\SO d)+\e^{1-\frac{d}r}(\|\nabla^2g_0\|_{L^r(A)}+\bar\rho)\right)\notag\\
 \lesssim&C(\bar \rho)\left(\bar\rho(1+\e^{1-\frac{d}{r}})+\e^{1-\frac{d}{r}}\|\nabla^2g_0\|_{L^r(A)}\right).
\end{align*}
Thus, for $\bar\rho\lesssim1$ and $\e_0 \lesssim (1+\|\nabla^2g_0\|_{L^r(A)})^{-\frac{r}{r-d}}$ sufficiently small, $\bar v_\e$ defined in \eqref{def:vepsl} satisfies \eqref{est:barve1} for $\e\in(0,\e_0)$. Next, we show that \eqref{claim:vel:h1} holds true for all $\e\in (0,\e_0)$. A direct calculation yields
\begin{eqnarray*}
 \nabla \bar v_\e-\nabla v_\e&=&\nabla(\e\eta(\phi(\tfrac{\cdot}{\e},(\nabla u_0)_\e)-\phi(\tfrac{\cdot}{\e},\nabla u_0))\\
                      &=&\eta(\nabla_y\phi(\tfrac{\cdot}{\e},(\nabla u_0)_\e)-\nabla_y\phi(\tfrac{\cdot}{\e},\nabla u_0))\\
                      &&+\e\eta(D\phi(\tfrac{\cdot}{\e},\nabla (u_0)_\e)[\partial_j(\nabla u_0)_\e]\otimes e_j-D\phi(\tfrac{\cdot}{\e},\nabla u_0)[\partial_j\nabla u_0]\otimes e_j)\\
                      &&+\e\big(\phi(\tfrac{\cdot}{\e},(\nabla u_0)_\e)-\phi(\tfrac{\cdot}{\e},\nabla u_0)\big)\otimes\nabla\eta.     
\end{eqnarray*}
Since $U_{\bar{\bar\rho}}\ni F\mapsto \phi(F)\in W^{1,\infty}(Y)$ is locally Lipschitz, we deduce that
\begin{eqnarray*}
  &&\int_A|\phi(\tfrac{x}\e, (\nabla u_0)_\e)-\phi(\tfrac{x}\e,\nabla u_0)|^2+\int_A|\nabla_y\phi(\tfrac{x}\e,(\nabla u_0)_\e)-\nabla_y\phi(\tfrac{x}\e,\nabla u_{0})|^2\,dx\\
  &\lesssim& C(\bar \rho) \int_A|(\nabla u_0)_\e-\nabla u_{0}|^2 \lesssim C(\bar\rho)\e^{2(1+\frac{d}2-\frac{d}{r})}(\|\nabla^2 g_0\|_{L^{r}(A)}+\Lambda(f,g,g_0))^2.
\end{eqnarray*}
Moreover, by \eqref{u0:trunc:4}
\begin{eqnarray*}
  & &\int_A|D\phi((\nabla u_0)_\e)[\partial_j (\nabla u_0)_\e]-D\phi(\nabla u_0)[\partial_j\nabla u_0]|^2\\
  &\lesssim& C(\bar \rho)\int_{\{(\nabla u_0)_\e\neq\nabla u_0\}} |\partial_j\nabla u_0|^2+|\partial_j (\nabla u_0)_\e|^2 \\
  &\lesssim& C(\bar\rho)|\{(\nabla u_0)_\e\neq\nabla u_0\}|^\frac{r-2}{r}(\|\nabla^2 u_0\|_{L^r(A)}^2+\|\partial_j(\nabla u_0)_\e-\partial_j\nabla u_0\|_{L^r(A)}^2)\\
  &\lesssim& C(\bar \rho)\e^{2(\frac{d}{2}-\frac{d}{r})}(\|\nabla^2 g_0\|_{L^{r}(A)}+\Lambda(f,g,g_0))^2.
\end{eqnarray*}
Combining the previous two displayed estimates, we obtain
\begin{eqnarray*}
  \|\nabla\bar v_\e-\nabla v_\e\|_{L^2(A)} &\lesssim& C(\bar\rho)\e^{1+\frac{d}2-\frac{d}{r}}(\|\nabla^2 g_0\|_{L^{r}(A)}+\Lambda(f,g,g_0)),
\end{eqnarray*}
and thus \eqref{claim:vel:h1} for $\e<\e_0<1$ (recall $r>d\geq2$).

\substep{2.2} Estimates for $\e\in [\e_0,1)$.

We claim that $\bar v_\e:= u_0$ satisfies \eqref{est:barve1} and \eqref{est:barve2} for $\e\in[\e_0,1)$ where $\e_0$ is as in Substep~2.1. Indeed, by assumption we have
$$ \|\dist(\nabla u_0,\SO d)\|_{L^\infty(A)}\leq \bar{\bar\rho}<\delta.$$
Moreover, it can be checked that
$$\|u_0-v_\e\|_{H^1(A)}\lesssim C(\bar\rho)(\Lambda(f,g,g_0)+\e(\Lambda(f,g,g_0)+\|\nabla^2g_0\|_{L^r(A)})),$$
and thus for all $1>\e\geq\e_0=C(A,\alpha,\beta,\mu,\delta)(1+\|\nabla^2 g_0\|_{L^r(A)})^{-\frac{r}{r-d}}$ we get
$$\|u_0-v_\e\|_{H^1(A)}\lesssim C(\bar\rho)\e(1+\|\nabla^2 g_0\|_{L^r(A)}^{\frac{r}{r-d}})(\Lambda(f,g,g_0)+\e(\Lambda(f,g,g_0)+\|\nabla^2g_0\|_{L^r(A)})).$$
\step 3 Conclusion.

Let $\bar\rho$ and $\bar v_\e$ be as in the previous step and $v_\e$ the usual two-scale expansion with cut-off. Let $\e\in(0,1)$. As in the heuristic argument at the beginning of this proof, a combination of the strong $\beta$-convexity of $V_\ho$, \eqref{WhgeqVh}, \eqref{pf:onecell} and \eqref{est:barve1} yields
\begin{eqnarray*}
  \frac{\beta}{4}\int_A|\nabla w_\e-\nabla v_\e|^2
  &\leq&   \frac\beta2\int_A|\nabla w_\e-\nabla u_\e|^2 +   \frac\beta2\int_A|\nabla v_\e-\nabla u_\e|^2\\
  &\leq& \mathcal E_\e(w_\e)-\mathcal E_\e(u_\e)+\mathcal E_\e(v_\e)-\mathcal E_\e(u_\e)\\
  &=& \mathcal E_\e(w_\e)-\mathcal E_\e(\bar v_\e)+\mathcal E_\e(\bar v_\e)+\mathcal E_\e(v_\e)-2\mathcal E_\e(u_\e)\\
  &\leq & \mathcal I_\e(w_\e)-\mathcal I_\e(\bar v_\e)+\mathcal E_\e(\bar v_\e)+\mathcal E_\e(v_\e)-2\mathcal E_\e(u_\e)\\
  &\leq & \left(\mathcal I_\e(w_\e)-\inf_{g+W_0^{1,p}}\mathcal I_\e\right)+\mathcal E_\e(\bar v_\e)+\mathcal E_\e(v_\e)-2\mathcal E_\e(u_\e).
\end{eqnarray*}
Next, we estimate the last three terms:
\begin{eqnarray*}
  &&\mathcal E_\e(\bar v_\e)+\mathcal E_\e(v_\e)-2\mathcal E_\e(u_\e)=\mathcal E_\e(\bar v_\e)-\mathcal E_\e(u_\e)+\mathcal E_\e(v_\e)-\mathcal E_\e(u_\e).
\end{eqnarray*}
Note that the last difference is estimated in Proposition~\ref{P:C}. For the first difference we proceed similarly:
\begin{align*}
  0\leq&\, \mathcal E_\e(\bar v_\e)-\mathcal E_\e(u_\e)\stackrel{\eqref{est:diffene}}{\leq} \tfrac1{2\beta} \|\nabla \bar v_\e-\nabla u_\e\|_{L^2(A)}^2\\
  \leq&\, \tfrac1\beta(\|\nabla \bar v_\e-\nabla v_\e\|_{L^2(A)}^2+\|\nabla  v_\e-\nabla u_\e\|_{L^2(A)}^2),\\
\intertext{We bound the first term on the right-hand side from above by appealing to \eqref{claim:vel:h1}, and the second term by appealing to Proposition~\ref{P:C}. Thus,}
0\leq&\, \mathcal E_\e(\bar v_\e)-\mathcal E_\e(u_\e)\\
\lesssim&\,C(\bar \rho)\bigg(\e^2 (1+\|\nabla^2 g_0\|_{L^r(A)})^{\frac{2r}{r-d}}(\|\nabla^2 g_0\|_{L^r(A)}+\Lambda(f,g,g_0))^2\\
   &\qquad\qquad+\e^2\|\nabla^2 u_0\|_{L^2(A)}^2 + \int_A\left(\e^2|\nabla\eta|^2+(\eta-1)^2\right)\dist^2(\nabla u_0(x),\SO d)\,dx\bigg)\\
  \lesssim&\, C(\bar \rho)\left(\e^2(1+\|\nabla^2 g_0\|_{L^r(A)})^{\frac{2r}{r-d}}(\|\nabla^2 g_0\|_{L^r(A)}+\Lambda(f,g,g_0))^2+\e \Lambda(f,g,g_0)^2\right).
\end{align*}
It is left to estimate the difference between the two-scale expansion with cut-off $v_\e$ with the expansion without cut-off, i.e.~$u_0+\e \phi(\tfrac{\cdot}\e,\nabla u_0)$. Appealing to the previous calculations, it is easy to see that
\begin{align*}
 \|v_\e - (u_0+\e\phi(\tfrac{\cdot}\e,\nabla u_0))\|_{L^2(A)}\leq 2\e \|\phi(\tfrac{\cdot}\e,\nabla u_0)\|_{L^2(A)}\lesssim C(\bar\rho) \e \Lambda(f,g,g_0),
\end{align*}
and
\begin{align*}
 &\|\nabla v_\e - \nabla (u_0+\e\phi(\tfrac{\cdot}\e,\nabla u_0))\|_{L^2(A)}\leq \e \|\nabla((\eta-1)\phi(\tfrac{\cdot}\e,\nabla u_0))\|_{L^2(A)}\\
 &\lesssim  \|\nabla \phi(\tfrac{\cdot}\e,\nabla u_0)\|_{L^\infty(A)}\|\eta-1\|_{L^2(A)}+\|\e\nabla\eta\|_{L^2(A)}\|\phi(\tfrac{\cdot}\e,\nabla u_0)\|_{L^\infty(A)}\\
 &\qquad+\e\|D\phi(\tfrac{\cdot}\e,\nabla u)\|_{L^\infty(A)}\|\nabla^2 u_0\|_{L^2(A)}\\
 &\lesssim C(\bar\rho)\left(\e^\frac12\Lambda(f,g,g_0)+\e (\|\nabla g_0\|_{L^r(A)}+\Lambda(f,g,g_0))\right), 
\end{align*}
which finally concludes the proof.

\medskip

\section{Layered Composite}\label{sec:layer}

\begin{proof}[Proof of Proposition~\ref{P:layer}]
By Lemma~\ref{L:wv} there exists $\beta,\delta,\mu>0$ and matching convex lower bounds $V_i\in \mathcal V_\beta$, $i=1,\dots,N$ such that 
\begin{eqnarray*}
 W_i(F)+\mu\det (F)&\geq& V_i(F)\quad\mbox{for all $F\in\R^{d\times d}$},\\
 W_i(F)+\mu\det (F)&=& V_i(F)\quad\mbox{for all $F\in U_\delta$. }
\end{eqnarray*}
As before, we denote by $\phi(F)$ and $\sigma(F)$ the extended corrector associated with $V$, cf.~Lemma~\ref{L:phi}. In Step~1 we prove that $F\mapsto\phi(F)$ is in $C^1(\R^{d\times d},W_{\per,0}^{1,\infty}(Y))$, $\phi(F)$ is one-dimensional, and $\sigma(F)=0$ for all $F\in\R^{d\times d}$. Moreover, we prove that
\begin{equation}\label{est:key}
 \dist(F,\SO d)\quad\mbox{sufficiently small implies}\quad \|\dist(F+\nabla \phi(F),\SO d)\|_{L^\infty(Y)}<\delta.
\end{equation}
As indicated in the discussion at the beginning of the proof of Theorem~\ref{T:1cell}, these regularity properties (which imply \eqref{ref:convexcorrector2a} and \eqref{ref:convexcorrector2b}) suffice to draw the conclusion of Theorem~\ref{T:1cell}. Similarly, a careful inspection of the proof of Theorem~\ref{T:NC} and Proposition~\ref{P:C} reveals that the higher regularity of the corrector, namely $F\mapsto \phi(F)$ is $C^1(U_{\bar{\bar\rho}},W_{\per,0}^{2,q}(Y))$ for some $q>d$, is used only to justify the chain rule \eqref{chainrule} for $\phi(\frac{\cdot}{\e},\nabla u_0)$ (see Proposition~\ref{P:C}, proof of Step~1). In Step~3, we prove that the chain rule \eqref{chainrule} is also valid in the present situation. Hence, the proofs of Theorem~\ref{T:NC} and Proposition~\ref{P:C} extend to the layered case discussed here, which completes the argument of the proposition.

\step{1} Regularity and one-dimensionality of $\phi$.

We claim that $F\mapsto \phi(F)$ is in $C^1(\R^{d\times d}, W_{\per,0}^{1,\infty}(Y))$ and $\sigma(F)=0$ for all $F\in\R^{d\times d}$. Moreover, for every $F\in\R^{d\times d}$ the corrector $\phi(F)$ depends only on $y_1$ and is affine on $(t_{i-1},t_i)\times (0,1)^{d-1}$ for all $i=1,\dots,N$.

This is well-known and can be seen as follows: Consider $\psi_z(y):=\phi(y+z,F)$, where $z=(0,z')$ for some $z'\in\R^{d-1}$. Since $\phi(F)\in H_{\per,0}^1(Y)$ solves \eqref{corrector:EL}, we obtain $\psi_z\in H_{\per,0}^1(Y)$ and for every $\eta\in H_\per^1(Y)$:
$$\int_YDV(y,F+\nabla \psi_z(F))[\nabla \eta]\,dy=\int_{Y+z}DV(y,F+\nabla \phi(F))[\nabla \eta(y-z)]\,dy=0,$$
where we use for the last equality that $\eta(\cdot-y)\in H_\per^1(Y)$ and \eqref{corrector:EL}. Hence $\psi_z$ solves the Euler-Lagrange equation \eqref{corrector:EL} and uniqueness of the corresponding minimizer yields $\phi_z=\phi(F)$. Since $z'\in\R^{d-1}$ is arbitrary, we obtain that $\phi(F)$ depends only on $y_1$, i.e.~$\phi(y,F)=\tilde\phi(y_1,F)$ for some $\tilde\phi(F)\in H_{\per,0}^1((0,1))$. In particular, $\tilde \phi(F)$ is characterised as the unique minimizer of the following one-dimensional minimization problem
\begin{equation*}
 \min_{\tilde \phi\in H_{\per,0}^1((0,1))}\int_0^1\sum_{i=1}^N\chi_i(t)V_i(F+ \tilde \phi'(t)\otimes e_1)\,dt.
\end{equation*}
By the convexity of $V_i$ for $i=1,\dots,N$, the minimizer of the above problem has to be affine on each segment $(t_{i-1},t_i)$. Hence, for every $j\in\{1,\dots,N\}$ there exist maps $c_j,d_j:\R^{d\times d}\to\R^d$ such that 
\begin{equation*}
 \phi(y,F)=\tilde\phi(y_1,F)=c_j(F)y_1+d_j(F)\quad\mbox{for all $y\in Y$ with $y_1\in (t_{j-1},t_j)$}.
\end{equation*}
This rigidity and the fact that $F\mapsto\phi(F)$ is in $C^1(\R^{d\times d},H_{\per,0}^1(Y))$, cf.~Lemma~\ref{L:Dphi}, yield $c_j,d_j\in C^1(\R^{d\times d},\R^d)$. Thus, $F\mapsto\phi(F)$ is in $C^1(\R^{d\times d},W_{\per,0}^{1,\infty}(Y))$. Furthermore, we have (the crude) estimate
\begin{equation}\label{eq:correctorlayer}
 \|\nabla \phi(F)\|_{L^\infty(Y)}\leq \max_{i=1,\dots,N}(t_i-t_{i-1})^{-1}\|\nabla \phi(F)\|_{L^2(Y)}.
\end{equation}
Since $V(\cdot,F)$ and $\phi(F)$ depend only on $x_1$, the Euler-Lagrange equation for $\phi(F)$ implies that the flux is constant, i.e there exists $c_{ij}\in \R$ such that $DV(x,F+\nabla \phi(x,F))[e_i\otimes e_j]=c_{ij}$ for all $x\in Y$ and $i,j\in\{1,\dots,d\}$ (in particular $c_{ij}=0$ for $j\geq2$). Thus \eqref{def:J} and \eqref{eq:sigma} imply that $\sigma=0$.

\step{2} Proof of estimate \eqref{est:key}.

We claim that there exists a constant $C$ depending only on $\beta$ such that
\begin{equation}\label{eq:correctorlayer1}
 \|\phi(F)\|_{H^1(Y)}\leq C\dist(F,\SO d).
\end{equation}
Note that a combination of \eqref{eq:correctorlayer} and \eqref{eq:correctorlayer1} implies \eqref{est:key}. Fix $R\in \SO d$. The strong convexity of $V$, the boundedness of $D^2V$, and $DV(\cdot,R)=D\det (R)$ imply that 
\begin{align*}
 0\geq& \int_YV(y,F+\nabla \phi(F))-V(y,R)+V(y,R)-V(y,F)\,dy\\
 \geq&\int_Y DV(y,R)[F-R+\nabla \phi(F)] +\frac{\beta}{2} |F-R+\nabla \phi(F)|^2-DV(y,F)[F-R]\,dy\\
 =&\int_Y DV(y,R)[F-R] -DV(y,F)[F-R] +\frac{\beta}{2} |F-R+\nabla \phi(F)|^2\,dy\\
 \geq&-(\frac1\beta+\frac{\beta}{2})|F-R|^2+\frac{\beta}{4}\|\nabla \phi(F)\|_{L^2(Y)}^2.
\end{align*}
The inequality \eqref{eq:correctorlayer1} follows by choosing $R$ such that $\dist(F,\SO d)=|F-R|$.

\step{3} Chain rule

Given $\nabla u_0\in W^{1,r}(A,\R^{d\times d})$ with $r>d$, we show that $\phi(\frac{\cdot}\e,\nabla u_0)\in W^{1,r}(A)$ and the chainrule \eqref{chainrule}.     

By Step~1, the map $(y,F)\mapsto \phi(y,F)$ satisfies
$$\phi(\cdot,F)\in W_{\per,0}^{1,\infty}(Y)\quad\mbox{for all $F\in\R^{d\times d}$ and }\phi(y,\cdot)\in C^1(\R^{d\times d})\quad\mbox{for almost every $y\in Y$}.$$
Hence, $x\mapsto\phi(\frac{x}\e,\nabla u_0(x))$ is measurable and $\phi(\frac{\cdot}{\e},\nabla u_0)\in L^\infty(A)$. Appealing to the characterization of Sobolev spaces by difference quotients, we show $\phi(\tfrac\cdot\e,\nabla u_0)\in W^{1,r}(A)$. By Sobolev embedding, we find $R<\infty$ such that $|\nabla u_0(x)|<R$ for almost every $x\in A$. For every $h>0$, $j=1,\dots,d$ and $A'\Subset A$ such that $he_j<\dist(A',\partial A)$, we have 
\begin{align*}
 &\tfrac1h\|\phi(\tfrac{\cdot+he_j}\e,\nabla u_0(\cdot+he_j))-\phi(\tfrac{\cdot}\e,\nabla u_0)\|_{L^p(A')}\\
 &\quad\leq \tfrac1h \sup_{|F|\leq R}\|D\phi(F)\|_{W^{1,\infty}(Y)}\|\nabla u_0(\cdot+h)- \nabla u_0\|_{L^p(A')}+\tfrac1\e\sup_{|F|\leq R}\|\nabla \phi(F)\|_{W^{1,\infty}(Y)}\\
 &\quad\leq \sup_{|F|\leq R}\|D\phi(F)\|_{W^{1,\infty}(Y)}\|\nabla^2u_0\|_{L^p(A)}+\tfrac1\e\sup_{|F|\leq R}\|\nabla \phi(F)\|_{W^{1,\infty}(Y)},
\end{align*}
which implies $\phi(\frac\cdot\e,\nabla u)\in W^{1,r}(A)$.

Finally, we prove that the chain rule \eqref{chainrule} holds for almost every $x\in A$. By Step~1, there exists $c_j,d_j\in C^1(\R^{d\times d},\R^d)$, $j\in\{1,\dots,N\}$, such that for $y\in (t_{j-1},t_j)\times (0,1)^{d-1}$ we have $\phi(y,F)=c_j(F)y_1+d_j(F)$. Hence, $(y,F)\mapsto \phi(y,F)$ is in $C^1((t_{j-1},t_j)\times (0,1)^{d-1}\times \R^{d\times d})$ for all $j=1,\dots,N$. Since $\nabla u_0\in W^{1,r}(A)$ with $r>d$, there exists a nullset $\mathcal N$ such that the classical derivative of $\nabla u_0$ exists in $A\setminus \mathcal N$ and thus the chain rule \eqref{chainrule} is valid pointwise for every 
$$x\in A\setminus (\mathcal N\cup \{\frac{x}{\e}\in A\, |\, \frac{x_1}{\e}= kt_j\,\mbox{for some $k\in \Z$ and $j\in\{1,\dots,N\}$}\}.$$
Since the exceptional set in the above formula is a nullset the claim is proven.
\end{proof}

\section{Acknowledgments}
SN would like to thank Stefan M\"uller for inspiring discussions, while working on the paper \cite{MN11}; some ideas of the present paper emerged in this process. This work was supported by the DFG in the context of TU Dresden's Institutional Strategy \textit{``The Synergetic University''}.

\appendix

\section{Appendix}\label{appendix}

We used at several places the following result on continuity and differentiability of composition operators.

\begin{lemma}\label{L:composition}
 Let $q>d$, $n\in\N$ and let $A\subset\R^d$, be a bounded Lipschitz-domain. Suppose that $f\in C^1(\overline A\times U)$ where $U\subset\R^n$ is open and consider the map $v\mapsto F(v)$, where $F(v)(x)=f(x,v(x))$. Then,
 \begin{itemize}
  \item $v\mapsto F(v)$ is a  continuous mapping from the subset $B:=\{v\in W^{1,q}(A,\R^n)\ |\ v(x)\in U\,\mbox{for all $x\in\bar A$}\}$ of $W^{1,q}(A,\R^n)$ into $W^{1,q}(A)$;
  \item if $f\in C^1(\overline A\times U)$ then $v\mapsto F(v)$ is $C^1(B,W^{1,q}(A))$ and the differential at any $v$ is given by
  \begin{equation*}
  DF(v)[w]=Df(\cdot,v)[w]\in W^{1,q}(A);
 \end{equation*}
 \end{itemize}
\end{lemma}
\begin{proof}
 This is a special case of \cite[Theorem 3.1 p.27, Theorem 4.1 p.32]{Valent}
\end{proof}

In the proof of Theorem~\ref{T:NC}, we used the following well-known version of Morrey's inequality, which we state for the readers convenience
\begin{lemma}\label{L:Morreycone}
Consider a bounded cone $C\subset\R^d$ with apeture angle $\kappa$, and height $h>0$. Given $p>d$ there exists $c=c(d,\kappa,p)$ such that for any $u\in W^{1,p}(\R^d)\cap C(\R^d)$ and $x,y\in \overline C$ holds
\begin{equation*}
 |u(x)-u(y)|\leq ch^{1-\frac{d}p}\|\nabla u\|_{L^p(C)}
\end{equation*}
\end{lemma}

\begin{proof}
Throughout the proof, we use $\lesssim$ whenever it holds $\leq$ up to a multiplicative positive constant which only depend on $d,\kappa$ and $p$. Recall that by the convexity of $C$ and $h^d\lesssim |C|$, we have for almost every $x\in C$ that
\begin{align*}
 |u(x)-\fint_{C} u|\leq& \frac{h^d}d\fint_C|x-z|^{1-d}|\nabla u(z)|\,dz\lesssim \int_C|x-z|^{1-d}|\nabla u(z)|\,dz,
\end{align*}
see e.g.~\cite[Lemma 7.16]{GT02}. By continuity of $u$, the above estimate is true for all $x\in \bar C$. Hence, appealing to standard estimates for the Riesz potential, see e.g.~\cite[Lemma 7.12]{GT02}, we obtain
\begin{align*}
 |u(x)-u(y)|\leq& |u(x)-\fint_{C} u|+|u(y)-\fint_{C} u|\\
 \lesssim&  \int_C|x-z|^{1-d}|\nabla u(z)|\,dz+ \int_C|y-z|^{1-d}|\nabla u(z)|\,dz\\
 \lesssim& |C|^{\frac1d-\frac{1}{p}}\|\nabla u\|_{L^p(C)}\lesssim h^{1-\frac{d}{p}}\|\nabla u\|_{L^p(C)},
\end{align*}
which proves the claim.
\end{proof}

\end{document}